\documentclass{amsart}

\usepackage{amsmath, amssymb, amsthm}
\usepackage{amsbsy}
\usepackage{bm}
\usepackage{graphicx}
\usepackage{pgfplots}
\usepackage{epstopdf}
\usepackage{tikz}
\usetikzlibrary{patterns, calc}
\usetikzlibrary{arrows,automata}
\usepackage{caption}
\usepackage{subcaption}
\usepackage{color}
\usepackage{fullpage}
\usepackage{hyperref}
\usepackage{tabularx}
\usepackage{enumerate}
\usepackage[utf8]{inputenc}
\usepackage{makecell}

\newcommand{\Bea}{\begin{eqnarray*}}
\newcommand{\Eea}{\end{eqnarray*}}
\newcommand{\bea}{\begin{eqnarray}}
\newcommand{\eea}{\end{eqnarray}}

\newcommand{\bs}{\backslash}

\newcommand{\C}{\mathbb{C}}
\newcommand{\D}{\mathbb{D}}
\newcommand{\E}{\mathbb{E}}
\newcommand{\N}{\mathbb{N}}
\newcommand{\R}{\mathbb{R}}
\newcommand{\Z}{\mathbb{Z}}
\newcommand{\Pb}{\mathbb{P}}

\newcommand{\CE}{\mathcal{E}}
\newcommand{\CF}{\mathcal{F}}

\newcommand{\CI}{\mathcal{I}}
\newcommand{\CL}{\mathcal{L}}

\newcommand{\CN}{\mathcal{N}}

\newcommand{\CV}{\mathcal{V}}

\newcommand{\Var}{\text{Var}}

\newcommand{\comment}[1]{}

\def\1{\textbf{1}}

\def\tr{\text{tr}}
\def\sinc{\text{sinc}}

\def\bs{\backslash}

\newtheorem{thm}{Theorem}[section]
\newtheorem{prop}[thm]{Proposition}
\newtheorem{lem}[thm]{Lemma}
\newtheorem{cor}[thm]{Corollary} 

\theoremstyle{definition}
\newtheorem{dfn}[thm]{Definition}
\newtheorem{rem}[thm]{Remark}
\newtheorem{con}[thm]{Condition}

\numberwithin{equation}{section}

\allowdisplaybreaks

\title[CLT for random band matrices]{CLT for non-Hermitian random band matrices with variance profiles}
\author{Indrajit Jana}
\address{School of Basic Sciences, IIT Bhubaneswar, India, 752050}
\email{ijana@iitbbs.ac.in}
\date{\today}

\begin{document}
\maketitle

\begin{abstract}
 	We show that the fluctuations of the linear eigenvalue statistics of a non-Hermitian random band matrix of increasing bandwidth $b_{n}$ with a continuous variance profile $w_{\nu}(x)$ converges to a $N(0,\sigma_{f}^{2}(\nu))$, where $\nu=\lim_{n\to\infty}(2b_{n}/n)\in [0,1]$ and $f$ is the test function. When $\nu\in (0,1]$, we obtain an explicit formula for $\sigma_{f}^{2}(\nu)$, which depends on $f$,  and variance profile $w_{\nu}$. When $\nu=1$, the formula is consistent with Rider and Silverstein (2006) \cite{rider2006gaussian}. We also independently compute an explicit formula for $\sigma_{f}^{2}(0)$ i.e., when the bandwidth $b_{n}$ grows slower compared to $n$. In addition, we show that $\sigma_{f}^{2}(\nu)\to \sigma_{f}^{2}(0)$ as $\nu\downarrow 0$.
\end{abstract}


\section{Introduction}

In this article, we consider the linear eigenvalue statistics of random non-Hermitian band matrices with a variance profile. Let $M$ be an $n\times n$ random non-Hermitian matrix and $\lambda_{i}(M);1\leq i\leq n$ be its eigenvalues. Define the empirical spectral measure (ESM) of $M$ as

\begin{align*}
\mu_{M}=\frac{1}{n}\sum_{i=1}^{n}\delta_{\lambda_{i}(M)},
\end{align*}
where $\delta_{x}$ is unit point mass at $x$. It was shown, in a series of papers, that if the entries of $M$ are i.i.d. random variables with zero mean and unit variance, then asymptotically $\mu_{M/\sqrt{n}}$ converges to the uniform density on the unit disc in $\C$ \cite{girko1985circular, bai1997circular, edelman1997probability, tao2008random, tao2010random}. However, if the entries are not identically distributed, the limiting law may be different. In particular, when the entries of the matrix are multiplied by some predetermined weights, the matrix is called a random matrix with a variance profile. Limiting ESM of such matrices were found in \cite{cook2018non}.

In an analogous way to classical probability, limiting ESM is the law of large numbers for random matrices. One might be interested in finding fluctuations of such convergence after proper scaling, which is the central limit theorem (CLT) in classical probability. In the case of random matrices, we would be studying CLT of the sequence of random measures (the ESMs). One way to study such object is by studying $\int f\;d\mu_{M}$ for some test function $f$. This brings the question from the space of random measures to the space of real/complex valued random variables. More precisely, we define the linear eigenvalue statistics of $M$ with respect to a test function $f$ as
\begin{align*}
\CL_{f}(M)&:=n\int f\;d\mu_{M}=\sum_{i=1}^{n}f(\lambda_{i}(M)),\\
\CL_{f}^{\Delta}(M)&:=\CL_{f}(M)-nf(0).
\end{align*}
We consider the limiting distribution of $\CL_{f}^{\Delta}(M)$. Limiting behavior of such quantities were studied in \cite{johansson1998fluctuations, soshnikov1998tracecentral, lytova2009central}, \cite[section 5]{bourgade2016fixed} for Hermitian matrices; and in \cite{anderson2006clt, li2013central, jana2014fluctuations, shcherbina2015fluctuations} for Hermitian band matrices.

 In this article, we consider non-Hermitian matrices $M$ whose entries are complex valued random variables. Distributional limit of such objects was found in \cite{nourdin_peccati_2010, bauerschmidt2016two, rider2004deviations, rider2006gaussian, rider2007noise}, which was later extended in \cite{ameur2011fluctuations, o2016central, bekerman2018clt, leble2018fluctuations}. CLT for polynomial $f$ and real valued $M$ were studied in \cite{nourdin_peccati_2010}. More recently, CLT for products of random matrices were established in \cite{coston2018gaussian, gorin2018gaussian}; and words of random matrices were studied in \cite{ dubach2019words}.

In both the cases \cite{rider2006gaussian, nourdin_peccati_2010},  the matrix $M$ was a full matrix without any variance profile. Recently, for polynomial test functions, it was shown that $\CL_{f}^{\Delta}(M)$ for random symmetric/non-symmetric matrices with a variance profile converges to $\CN(0,\sigma_{f}^{2})$ in total variation norm \cite{adhikari2019linear}. However, since the results in \cite{adhikari2019linear} were stated in a very broad context, the exact expression of $\sigma^{2}_{f}$ was difficult to find. 

The main contribution of this article is calculating $\sigma_{f}^{2}$ for random band matrices with a variance profile. In \cite{rider2006gaussian}, the variance was calculated in the process of proving the CLT. The same procedure does not yield the variance in our case. So, the proof of CLT and calculation of variance is done using two separate methods. An $n\times n$ non-periodic (and periodic) band matrix of bandwidth $b_{n}$ is obtained by keeping $2b_{n}$ many off-diagonal vectors around the main diagonal (and around the corners), and setting rest of the off-diagonal vectors to zero.

\begin{figure}[h!]
	\centering
	\begin{subfigure}{0.45\textwidth}
		\centering
		\begin{tikzpicture}[scale=4.0]
		\draw (0,0) -- (1,0) -- (1,1) -- (0,1) -- (0,0);
		\draw [pattern=north west lines, pattern color=blue] (1, 0.75) -- (1,1) -- (0.75,1) -- (1,0.75);
		\draw [pattern=north west lines, pattern color=blue] (1, 0) -- (1,0.25) -- (0.25,1) -- (0, 1) -- (0, 0.75) -- (0.75, 0) -- (1, 0);
		\draw [pattern=north west lines, pattern color=blue] (0, 0) -- (0.25,0) -- (0, 0.25) -- (0, 0);
		\draw [<->] (0, 1.05) -- (0.25, 1.05)
		node  at (0.12, 1.1){$b_{n}$};
		\draw [<->] (0.75, 1.05) -- (1, 1.05) 
		node  at (0.87, 1.1){$b_{n}$};
		\end{tikzpicture}
		\caption{Periodic band matrix\\ of bandwidth $b_{n}$}
	\end{subfigure}
	\begin{subfigure}{0.45\textwidth}
		\centering
		\begin{tikzpicture}[scale=4.0]
		\draw (0,0) -- (1,0) -- (1,1) -- (0,1) -- (0,0);
		\draw [pattern=north west lines, pattern color=blue] (1, 0) -- (1,0.25) -- (0.25,1) -- (0, 1) -- (0, 0.75) -- (0.75, 0) -- (1, 0);
		\draw [<->] (0, 1.05) -- (0.25, 1.05)
		node at (0.12, 1.1){$b_{n}$};
		\end{tikzpicture}
		\caption{Non-periodic band matrix\\ of bandwidth $b_{n}$}
	\end{subfigure}
\caption{Blue lines represent the non-zero diagonal vectors. }
\end{figure}
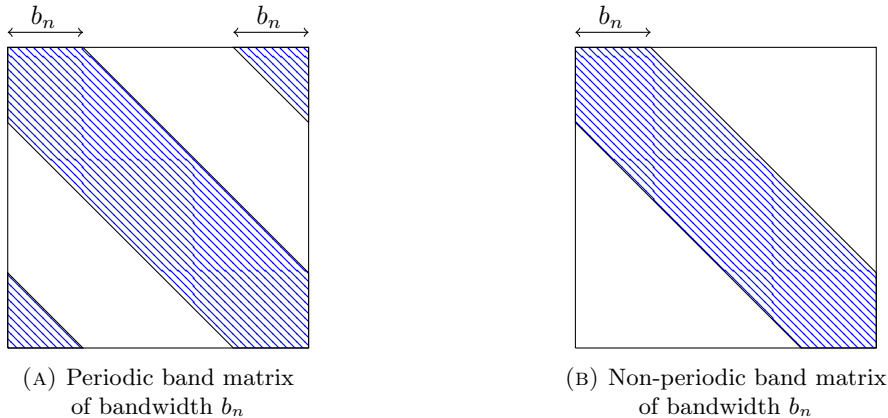

A precise definition of random band matrix is given in the Definition \ref{defn: band matrix with a variance profile}. In particular, we show that if we have a periodic band matrix with $2b_{n}+1=n$, then our results are consistent with that of \cite{rider2006gaussian}. In this context, we would also like to mention that while in full matrix case the unscaled $\CL_{f}^{\Delta}(M)$ converges to a Gaussian distribution, in band case we need to scale it by $\sqrt{(2b_{n}+1)/n}$. This shows a significant difference in between full and band matrices. In the first case $\Var(\CL_{f}^{\Delta}(M))$ remains constant, while in the latter case it grows as $O(n/b_{n})$.

The article is organized as follows. In Section \ref{sec: Preliminaries}, we enlist the notations and definitions. The main theorem is formulated in Section \ref{sec: main theorem}, and the proof is given in Section \ref{sec: proof of the main theorem}. In the process of the proof, we need the norm of the random matrix to be bounded almost surely, which is discussed in appendix \ref{sec: norm of the matrix}. Variance of the limiting distribution is calculated in Section \ref{sec: variance calculation}.

\textbf{Acknowledgment:} The author gratefully acknowledges many technical discussions with Brian Rider. The author also conveys thanks to Alexander Soshnikov; and Kartick Adhikari, Koushik Saha for mathematical discussions and sponsoring a visit to UC Davis; and providing constructive feedbacks about the manuscript. In addition, the author appreciates many constructive feedback provided by the referees.

\section{Preliminaries and notations}\label{sec: Preliminaries} For convenience, we do not indicate the size of a matrix in its name. For example, to denote an $n\times n$ matrix $A$, we simply write $A$ instead of $A_{n}$. In addition, throughout this paper we use the following notations;
\begin{align*}
&\{e_{1},e_{2},\ldots,e_{n}\} \;\text{is the cannonical basis of $\C^{n}$}\\
&I_{j}\;\;\text{is a band index set as defined in the Definition \ref{defn: band matrix with a variance profile}}\\
&\CI_{j}\;\;\text{is a band diagonal matrix as defined in the Definition \ref{defn: band matrix with a variance profile}}\\
&a_{ij}:=\text{$ij$th entry of the matrix $A$}\\
&a_{k}:=\text{$k$th column of the matrix $A$}\\
&A^{(\Lambda)}:=\text{The matrix $A$ after setting columns which are indexed by $\Lambda$ to zero}\\
&A^{*}:=\text{Complex conjugate transpose of the matrix $A$}\\
&R_{z}(A):=(zI-A)^{-1},\;\;\text{if $A$ is a square matrix}\\
&\tr R_{z}^{\Delta}(A):=\tr R_{z}(A)-\frac{n}{z}\\
&\lambda_{1}(A),\lambda_{2}(A),\ldots, \lambda_{n}(A) \;\;\text{are the eigenvalues of an $n\times n$ matrix $A$}\\
&\mu_{A}:=\frac{1}{n}\sum_{i=1}^{n}\delta_{\lambda_{i}(A)},\;\;\text{for an $n\times n$ matrix $A$}\\
&\CL_{f}(A):=\sum_{i=1}^{n}f(\lambda_{i}(A))\\
&\CL_{f}^{\Delta}(A):=\CL_{f}(A)-nf(0)\\
&\|v\|_{p}:=\left(\sum_{k=1}^{n}|v_{k}|^{p}\right)^{1/p} \;\text{for}\;1\leq p<\infty, \;\text{and}\; \|v\|_{\infty}=\max_{k}|v_{k}|,\;\text{where}\;v\in \C^{n}\\
&\|A\|:=\sup_{\|v\|_{2}=1}\|Av\|_{2};\;\text{operator norm of the matrix $A$}\\
&\D_{r}:=\{z\in \C:|z|\leq r\}\\
&\partial\D_{r}:=\{z\in \C:|z|=r\}\\
&\CN_{k}(\mu, \Sigma, \Upsilon):=\text{$k$ dimensional multivatiate normal distribution}\\
	&\;\;\; \text{with mean $\mu$, covariance matrix $\Sigma$, and pseudo-covariance matrix $\Upsilon$}\\
&\CE_{k}[\xi]:=\text{Expectation of the random variable $\xi$ with respect to the}\\
	&\;\;\; \text{$k$th column of the underlying random matrix $M$}\\
&\E_{k}[\xi]:=\CE_{k+1}\CE_{k+2}\ldots\CE_{n}[\xi]. \text{ In particular, $\E_{0}[\xi]=\E[\xi]$ and $\E_{n}[\xi]=\xi$} \\
&\xi^{\circ}:=\xi-\E[\xi]\\
&\xi^{\circ}_{k}:=\xi-\CE_{k}[\xi]\\
&\xi_{n}\stackrel{d}{\to}\xi  \text{ denotes that $\xi_{n}$ converges to $\xi$ in distribution}\\
&\xi_{n}\stackrel{p}{\to}\xi  \text{ denotes that $\xi_{n}$ converges to $\xi$ in probability}\\
&\xi_{n}\stackrel{a.s.}{\to}\xi  \text{ denotes that $\xi_{n}$ converges to $\xi$ almost surely}\\
&a_{n}=\Omega(b_{n})\;\;\text{denotes that there exists $\;c>0$ such that $a_{n}\geq cb_{n}$ for all $n$}
\end{align*}
\begin{dfn}[Band matrix with a variance profile]\label{defn: band matrix with a variance profile} Let $\nu\in (0,1]$, and $w:\R\to[0,\infty)$ be a piece-wise continuous function supported on $[-1/2,1/2]$ such that it is continuous at $0$ and  $\int_{-1/2}^{1/2}w(x)\;dx=1$. Define a $1/\nu$ periodic function $w_{\nu}$ and a non-periodic function $w_{0}$ as follows 
\begin{align*}
w_{\nu}(x)&=w(x), \;\;\forall\;x\in \left[-\frac{1}{2\nu},\frac{1}{2\nu}\right],\;\;\;\;\&\;\;w_{\nu}\left(x+\frac{1}{\nu}\right)=w_{\nu}(x)\\
w_{0}(x)&=w(x),\;\;\forall\;x\in \R.
\end{align*}
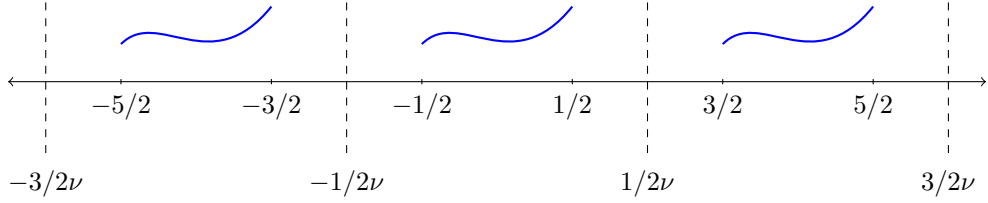
\begin{figure}[h!]
	\begin{tikzpicture}
	\draw [<->] (-6.5,0) -- (6.5, 0);
	\foreach \x in {-3,-1,1,3}
	\draw [dashed] (2*\x cm,30pt) -- (2*\x cm,-30pt) node[anchor=north] {$\x/2\nu$};
	\foreach \x in {-5,-3,-1,1,3,5}
	\draw (\x cm,1pt) -- (\x cm,-1pt) node[anchor=north] {$\x/2$};
	\draw [thick, blue] (-5,0.5) .. controls (-4.5,1) and (-3.8,0) .. (-3,1);
	\draw [thick, blue] (-1,0.5) .. controls (-0.5,1) and (0.2,0) .. (1,1);
	\draw [thick, blue] (3,0.5) .. controls (3.5,1) and (4.2,0) .. (5,1);
	\end{tikzpicture}
	\caption{$w_{\nu}$ is constructed by copying and pasting $w(x)\1_{[-1/2\nu,1/2\nu]}$ in each $[(2k-1)/2\nu, (2k+1)/2\nu]$. This particular figure is for $\nu=1/2$.}
\end{figure}
In particular, $w_{\nu}(x)\stackrel{\nu\to 0}{\to}w_{0}(x)$ on any compact subset of $\R$. Let $\{x_{ij}:1\leq i,j\leq n\}$ be a set of i.i.d. random variables for each $n$, and $b_{n}=(c_{n}-1)/2$.
\begin{enumerate}
\item When $\lim_{n\to\infty}(c_{n}/n)=\nu\in (0,1]$, define a periodic band matrix $M_{\nu}^{\circledcirc}$ of bandwidth $b_{n}$ as
\begin{align*}
m_{ij}=\frac{1}{\sqrt{c_{n}}}x_{ij}\sqrt{w_{\nu}\left(\frac{i-j}{c_{n}}\right)}.
\end{align*}

\item When $c_{n}=o(n)$, define non-periodic band matrix $M_{0}^{\Large{\oslash}}$ of bandwidth $b_{n}$ as
\begin{align*}
m_{ij}=\frac{1}{\sqrt{c_{n}}}x_{ij}\sqrt{w_{0}\left(\frac{i-j}{c_{n}}\right)},
\end{align*}
and a periodic band matrix $M_{0}^{\circledcirc}$ of bandwidth $b_{n}$ as 
\begin{align*}
m_{ij}&=\frac{1}{\sqrt{c_{n}}}x_{ij}\sqrt{w_{0}\left(\frac{i-j}{c_{n}}\right)}
+\frac{1}{\sqrt{c_{n}}}x_{ij}\sqrt{w_{0}\left(\frac{i-j+n}{c_{n}}\right)}
+\frac{1}{\sqrt{c_{n}}}x_{ij}\sqrt{w_{0}\left(\frac{i-j-n}{c_{n}}\right)}.
\end{align*}
\end{enumerate}

In this context, let us also define the band index set $I_{j}$, and the band diagonal matrix $\CI_{j}$ as follows
\begin{align*}
I_{j}&:=\left\{\begin{array}{ll}
\{1\leq i\leq n: \min\{|i-j|, n-|i-j|\}\leq b_{n}\}\, & \text{if $M$ is periodic}\\
\{1\leq i\leq n:|i-j|\leq b_{n}\}, & \text{if $M$ is non-periodic}
\end{array}
\right.\\
\CI_{j}&:=\left\{\begin{array}{ll}
\sum_{i\in I_{j}}e_{i}e_{i}^{t}w_{\nu}\left(\frac{i-j}{c_{n}}\right), & \text{if $\nu\in (0,1]$}\\
\sum_{i\in I_{j}}e_{i}e_{i}^{t}w_{0}\left(\frac{i-j}{c_{n}}\right), & \text{if $\nu=0$}.
\end{array}
\right.
\end{align*}
In particular, we observe that if $\E[x_{ij}]=0$ and $\E[|x_{ij}|^{2}]=1$ for all $1\leq i,j\leq n$, then
\begin{align*}
\E[m_{j}m_{j}^{*}]=\frac{1}{c_{n}}\CI_{j},
\end{align*}
where $m_{j}$ is the $j$th column of $M$, which is one of the $M_{\nu}^{\circledcirc}, M_{0}^{\circledcirc}, M_{0}^{\oslash}$ in the above.
\end{dfn}

In the above definition, we notice that if we take $w(x)\equiv 1$, then it yields band matrices without any variance profile i.e., identical variances. We also observe that if the periodic band matrix is a full matrix then $\nu=1$.

We would also like to mention that when $\nu\in (0,1]$, we are considering only periodic band matrices; and when $\nu=0$, we are considering both periodic and non-periodic band matrices. In short, we are not considering non-periodic band matrices when $\nu\in(0,1]$. The CLT may still be true for non-periodic band matrices with $\nu\in (0,1]$. However, our method of variance calculation does not work in this case; as outlined in Remarks \ref{rem: why periodic2}, \ref{rem: why periodic}.

\begin{dfn}[Poincar\'e inequality]\label{defn: poincare inequality}
	A complex random variable $\xi$ is said to satisfy Poincar\'e inequality with constant $\alpha\in (0,\infty)$ if for any differentiable function $h:\C\to\C$, we have $\Var(h(\xi))\leq \frac{1}{\alpha}\E[|h'(\xi)|^{2}]$. Here $\C$ is identified with $\R^{2}$.
\end{dfn}

Here are some properties of Poincar\'e inequality
\begin{enumerate}
	\item If $\xi$ satisfies Poincar\'e inequality with constant $\alpha$, then $c\xi$ also satisfies Poincar\'e inequality with constant $\alpha/ c^{2}$ for any $c\in \R\bs \{0\}$.
	\item If two independent random variables $\xi_{1},\xi_{2}$ satisfy the Poincar\'e inequality with the same constant $\alpha$, then for any differentiable function $h:\C^{2}\to\C$, $\Var(h(\xi))\leq \frac{1}{\alpha}\E[\|\nabla h\|_{2}^{2}]$ i.e., $\xi:=(\xi_{1},\xi_{2})$ also satisfies Poincar\'e inequality with the same constant $\alpha$.
	\item\label{Poincare exponential tail bound} \cite[Lemma 4.4.3]{anderson2010introduction} If $\xi\in \C^{N}$ satisfies Poincar\'e inequality with constant $\alpha$, then for any differentiable function $h:\C^{N}\to\C$
	\begin{align*}
	\Pb(|h(\xi)-\E[h(\xi)]|>t)\leq K\exp\left\{-\frac{\sqrt{\alpha}t}{\sqrt{2}\|\|\nabla h\|_{2}\|_{\infty}} \right\},
	\end{align*}
	where $K=-2\sum_{i=1}^{\infty}2^{i}\log(1-2^{-2i-1})$. Here $\C^{N}$ is identified with $\R^{2N}$. In the above, $\|\nabla h(x)\|_{2}$ denotes the $\ell^{2}$ norm of the $N$-dimensional vector $\nabla h(x)$ at $x\in \C^{N}$; and $\|\|\nabla h\|_{2}\|_{\infty}=\sup_{x\in \C^{N}}\|\nabla h(x)\|_{2}$. In particular if $h:\C^{N}\to \C$ is a Lipschitz function with lipschitz constant $\|h\|_{Lip}$, then  $\|\|\nabla h\|_{2}\|_{\infty}\leq \|h\|_{Lip}$.
\end{enumerate}

For example, Gaussian random variables and compactly supported continuous random variables satisfy Poincar\'e inequality.

\section{Main result}\label{sec: main theorem}

\begin{con}\label{con: The condition}
	Let $\{x_{ij}:1\leq i,j\leq n\}$ be an i.i.d. set of complex-valued continuous random variables, and $M$ be one of the random band matrices as defined in the Definition \ref{defn: band matrix with a variance profile}. Assume that 
	\begin{enumerate}[(i)]
		\item $\E[x_{ij}]=0$ and $\E[|x_{ij}|^{2}]=1$ for all $1\leq i,j\leq n$,
		\item $x_{ij}$s are continuous random variables with bounded density and satisfy the Poincar\'e inequality with some universal constant $\alpha$. In particular, $\E[|x_{ij}|^{k}]\leq Ck^{2k}$ for all $k\geq 2$ and for some universal constant $C$.
		\item Either of the following is true.
		\begin{enumerate}[(a)]
			\item $n\geq c_{n}\geq \log^{3}n$, and $\E[x_{ij}^{k}]=0$ for all $1\leq i,j\leq n$; $k\in \N$,
			\item $\lim_{n\to\infty}(c_{n}/n)=\nu\in (0,1]$ and $\E[x_{ij}^{2}]=0$ for all $1\leq i, j\leq n$.
		\end{enumerate}
	\end{enumerate}
\end{con}

Here the Poincar\'e inequality is assumed for both $c_{n}=o(n)$ as well as $c_{n}=\Omega(n)$ to unify the proof. However in the latter case i.e.,  $c_{n}=\Omega(n)$, the proof may go through using the techniques of \cite{rider2006gaussian} without Poincar\'e inequality. Also in (iii)(a), it suffices to take $c_{n}>\log^{2+\epsilon}n$ (see \eqref{eqn: Estimates of moments of delta}) for some $\epsilon>0$. But we assume $c_{n}>\log^{3}n$ for simplicity.

The above conditions implies that $\|M\|\leq \rho$ for some fixed $\rho$ almost surely as $n\to\infty$. We shall discuss this in more details in appendix \ref{sec: norm of the matrix}.

\begin{thm}\label{thm: main theorem}
Let $M$ be one of the $n\times n$ random band matrices of bandwidth $b_{n}$ as defined in the Definition \ref{defn: band matrix with a variance profile} such that condition \ref{con: The condition} holds. Let $f_{1},f_{2},\ldots,f_{k}:\C\to \C$ be analytic on $\D_{\rho+\tau}$ for some $\tau>0$ and bounded elsewhere. 

Then as $n\to\infty$, 

$$\sqrt{\frac{c_{n}}{n}}(\CL_{f_{1}}^{\Delta}(M),\CL_{f_{2}}^{\Delta}(M),\ldots,\CL_{f_{k}}^{\Delta}(M))\stackrel{d}{\to}\CN_{k}(0,\Sigma, \Upsilon),$$ 
where $\Upsilon=0$,

\begin{align*}
\Sigma_{ij}&=-\frac{\nu}{4\pi^{2}}\sum_{k\in \Z}\oint_{\partial \D_{1}
}\oint_{\partial \D_{1}}\frac{f_{i}(z)\overline{f_{j}(\eta)}\hat{w}_{\nu}(k)}{(z\bar{\eta}-\hat{w}_{\nu}(k))^{2}}\;dz\;d\bar{\eta} \\
&\text{if $\nu=\lim_{n\to\infty}(c_{n}/n)\in (0,1]$ and $M$ is periodic},
\end{align*}
and
\begin{align*}
\Sigma_{ij}&=-\frac{1}{4\pi^{2}}\int_{\R}\oint_{\partial \D_{1}
}\oint_{\partial \D_{1}}\frac{f_{i}(z)\overline{f_{j}(\eta)}\hat{w}_{0}(t)}{(z\bar{\eta}-\hat{w}_{0}(t))^{2}}\;dz\;d\bar{\eta}\;dt \\
&\text{if $\lim_{n\to\infty}(c_{n}/n)=0$.}
\end{align*}
Here
\begin{align*}
 \hat{w}_{\nu}(k)&=\int_{-1/2\nu}^{1/2\nu}w_{\nu}(x)e^{2\pi ikx\nu}\;dx=\int_{-1/2}^{1/2}w(x)e^{2\pi i kx\nu}\;dx,\\
 \hat{w}_{0}(t)&=\int_{\R}e^{2\pi itx}w_{0}(x)\;dx=\int_{-1/2}^{1/2}e^{2\pi itx}w(x)\;dx.
\end{align*}
\end{thm}

 We give the proof in Section \ref{sec: proof of the main theorem} and variance is calculated in Section \ref{sec: variance calculation}. Before going into the proof, we would like to make some remarks about the above theorem. First of all, the theorem is stated for $\CL_{f}^{\Delta}(M)=\CL_{f}(M)-nf(0)$; not for $\CL_{f}^{\circ}(M):=\CL_{f}(M)-\E[\CL_{f}(M)]$. To the best of our knowledge, the circular law for random band matrices is not known for $c_{n}=o(n)$. If the circular law is true for random band matrices, we would asymptotically have
 \begin{align*}
 \lim_{n\to\infty}\frac{1}{n}\E[\CL_{f}(M)]=&\frac{1}{\pi}\int_{\D_{1}}f(z)\;d\Re(z)\;d\Im(z)
 =f(0)
 =\frac{1}{2\pi i}\int_{\partial \D_{\rho+\tau}}\frac{f(z)}{z}\;dz.
 \end{align*}
 
 Then asymptotically we would have $\CL_{f}^{\Delta}(M)\approx\CL_{f}^{\circ}(M)$. 
 
 Secondly, if $M_{\nu}^{\circledcirc}$ is a non-Hermitian full matrix with a continuous variance profile $w(x)$, then $\nu=1$. Limiting ESM of such matrices was discovered in \cite{cook2016limiting}. The following corollary provides a CLT for such matrices.
\begin{cor}
Let $M$ be a non-Hermitian random matrix with a variance profile $w(x)$ as defined in Definition \ref{defn: band matrix with a variance profile} such that condition \ref{con: The condition} holds. Let $f_{1},f_{2},\ldots,f_{k}:\C\to \C$ be analytic on $\D_{1+\tau}$ for some $\tau>0$ and bounded elsewhere. 

Then $(\CL_{f_{1}}^{\Delta}(M),\CL_{f_{2}}^{\Delta}(M),\ldots,\CL_{f_{k}}^{\Delta}(M))\stackrel{d}{\to}\CN_{k}(0,\Sigma, \Upsilon)$ as $n\to\infty$, 

where $\Upsilon=0$,
\begin{align*}
\Sigma_{ij}&=-\frac{1}{4\pi^{2}}\sum_{k\in \Z}\oint_{\partial \D_{1}
}\oint_{\partial \D_{1}}\frac{f_{i}(z)\overline{f_{j}(\eta)}\hat{w}(k)}{(z\bar{\eta}-\hat{w}(k))^{2}}\;dz\;d\bar{\eta},
\end{align*}

\begin{align*}
\hat{w}(k)=\int_{-1/2}^{1/2}w(x)e^{2\pi i kx}\;dx.
\end{align*}
\end{cor}
It should be noted that we can also have $\nu=1$ without having a full matrix; for example, by replacing $o(n)$ many off diagonals of a full matrix by zeros. The above corollary along with Theorem \ref{thm: main theorem} asserts that the limiting Gaussian distribution will be unchanged by doing so. In addition to the above corollary, we discuss a few more particular cases.

\begin{enumerate}[(I)]
	\item If we have the full matrix with i.i.d. entries, then $\nu=1$, and $w(x)\equiv1$. In that case, $\hat{w}(k)=\1_{\{k=0\}}$. As a result,

\begin{align*}
\Sigma_{ij}&=-\frac{1}{4\pi^{2}}\oint_{\partial \D_{1}}\oint_{\partial \D_{1}}\frac{f_{i}(z)\overline{f_{j}(\eta)}}{(1-z\bar{\eta})^{2}}\;dz\;d\bar{\eta}\\
&=-\frac{1}{4\pi^{2}}\oint_{\partial \D_{1}}\oint_{\partial \D_{1}}f_{i}(z)\overline{f_{j}(\eta)}\left[\frac{1}{\pi}\int_{\D_{1}}\frac{d\Re(\zeta)\;d\Im(\zeta)}{(\zeta-z)^{2}(\bar{\zeta}-\bar{\eta})^{2}}\right]\;dz\;d\bar{\eta}\\
&=\frac{1}{\pi}\int_{\D_{1}}f_{i}'(\zeta)\overline{f_{j}'(\zeta)}\;d\Re(\zeta)\;d\Im(\zeta).
\end{align*}

The above is the same as the expression obtained in \cite{rider2006gaussian}. In particular, if $f(z)=z^{k}$, then $\Var(\lim_{n\to\infty}\CL_{f}^{\Delta}(M))=k$ for all $k\in \N$. A numerical evidence of this fact is outlined in table \ref{tab: Numerical table}.

\item Let $\nu\in (0,1]$ and $M_{\nu}^{\circledcirc}$ be a periodic band matrix as defined in Definition \ref{defn: band matrix with a variance profile} with variance profile $w(x)\equiv 1$. Consider the monomial test function $f(z)=z^{l}$. Then as a consequence of Theorem \ref{thm: main theorem}, we have $\sqrt{c_{n}/n}\CL_{f}^{\Delta}(M)\stackrel{d}{\to}\CN(0,\CV_{\nu}(f),0)$, where
\begin{align*}
\CV_{\nu}(f)=l\nu \sum_{k\in \Z}\sinc^{l}(k\pi\nu),
\end{align*} 

and $\sinc(t)=\sin(t)/t$. The above equality follows from the fact that
\begin{align*}
\hat{w}_{\nu}(k)&=\int_{-1/2\nu}^{1/2\nu}w_{\nu}(x)e^{2\pi i k x\nu}\;dx
=\int_{-1/2}^{1/2}e^{2\pi i k x\nu}\;dx=\sinc\left(k\pi\nu\right).
\end{align*}

\item Let $f(z)=z^{l}$ and $\nu=0$, then $\sqrt{c_{n}/n}\CL_{f}^{\Delta}(M)\stackrel{d}{\to}\CN(0,\CV_{0}(f),0)$ where
\begin{align}\label{eqn: sinc integration is Irwin-Hall}
\CV_{0}(f)=&\frac{l}{\pi}\int_{\R}\sinc^{l}(t)\;dt
=\frac{l}{(l-1)!}\sum_{i=0}^{\lfloor l/2\rfloor}(-1)^{i}\binom{l}{i}\left(\frac{l}{2}-i\right)^{l-1}.
\end{align}

The above follows from the second part of the Theorem \ref{thm: main theorem} and the fact that 
\begin{align}\label{eqn: w0 is sinc for purely band}
\hat{w}_{0}(t)=\int_{-1/2}^{1/2}e^{2\pi i tx }\;dx=\sinc(\pi t).
\end{align}
The last equality of \eqref{eqn: sinc integration is Irwin-Hall} was obtained by establishing a connection to Irwin-Hall distribution which is outlined in Section \ref{subsec: irwin hall distribuion}. In addition if $l$ is even, one can also write $$\CV_{0}(f)=\frac{l}{(l-1)!}A(l-1, l/2-1),$$ where $A(n,m)$ is an Eulerian number. In combinatorics, $A(n,m)$ counts the number of permutations of the numbers $1,2,\ldots, n$ in which exactly $m$ elements are greater then the previous element.

\item 
We have the following table regarding integrals of the $\sinc(\cdot)$ function \cite{medhurst1965evaluation}.
\begin{align}\label{eqn: sinc integral table}
&\frac{1}{\pi}\int_{\R}\sinc(t)\;dt=1
&&\frac{1}{\pi}\int_{\R}\sinc^{2}(t)\;dt=1\nonumber\\
&\frac{1}{\pi}\int_{\R}\sinc^{3}(t)\;dt=\frac{3}{4}
&&\frac{1}{\pi}\int_{\R}\sinc^{4}(t)\;dt=\frac{2}{3}\nonumber\\
&\frac{1}{\pi}\int_{\R}\sinc^{5}(t)\;dt=\frac{115}{192}
&&\frac{1}{\pi}\int_{\R}\sinc^{6}(t)\;dt=\frac{11}{20}.
\end{align}

From the above equations, we obtain

\begin{align*}
\CV_{0}(z)&=1,\;\;\CV_{0}(z^{2})=2,\;\;\CV_{0}(z^{3})=\frac{9}{4}=2.25\\
\CV_{0}(z^{4})&=\frac{8}{3}\approx 2.67,\;\;\CV_{0}(z^{5})=\frac{575}{192}\approx 2.9948\\
\CV_{0}(z^{6})&=\frac{33}{10}=3.3\\
&\ldots
\end{align*}

Figure \ref{fig: band figures} shows some simulations done in Python. We have taken $n\times n$ periodic random band matrix with i.i.d. complex $N(0,1)$ variables. We run the simulation for test functions $f(z)=z^{l}$ and varying bandwidth.

\begin{center}

\begin{table}[h]
	$n=4000$\\
	\vspace{0.2cm}
\begin{tabular}{|c|c|c|c|c|c|c|}
	\hline
	$f(z)$ & $z$ & $z^{2}$ & $z^{3}$ & $z^{4}$ & $z^{5}$ & $z^{6}$\\ \hline
	$b_{n}=n^{0.3}$ & 1.0018 & 2.0357 & 2.4068 & 2.6585 & 3.0184& 3.3746\\ \hline
	$b_{n}=n/2$ & 1.0767 & 1.8385 & 3.1995 & 3.8756 & 4.6096 & 5.9900\\ \hline
\end{tabular}
\caption{Numerical values are the variances of $\sqrt{c_{n}/n}\CL_{f}^{\Delta}(M)$, calculated from $500$ iterations in each case.}\label{tab: Numerical table}
\end{table}
\end{center}

\begin{figure}[h!]
	\begin{subfigure}{0.3\textwidth}
		\includegraphics[width=\linewidth]{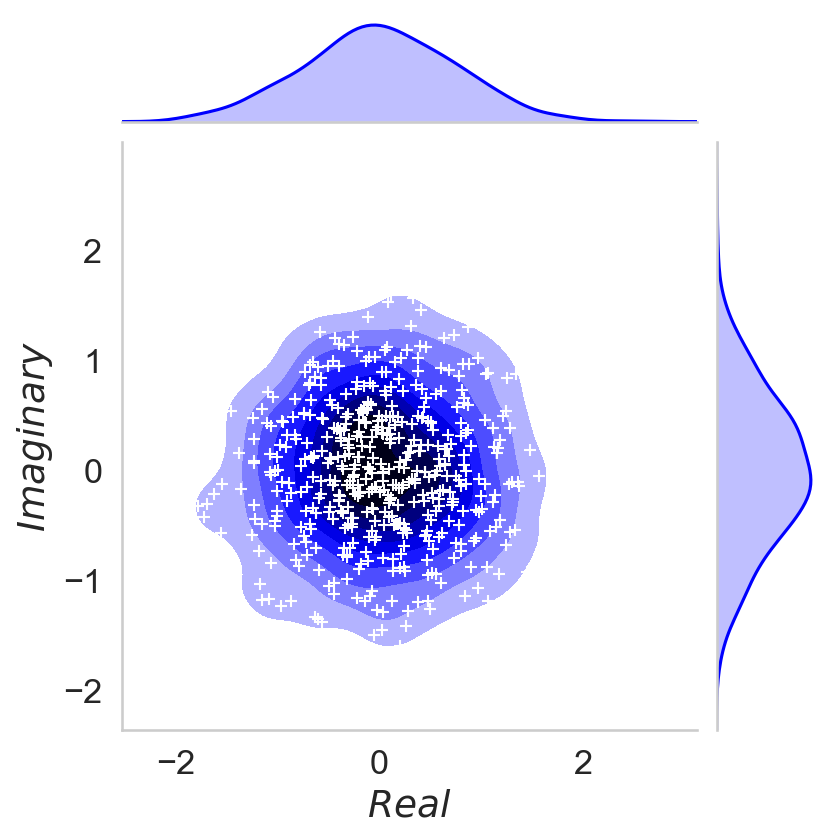}
		\caption{$f(z)=z$}
	\end{subfigure}
	\begin{subfigure}{0.3\textwidth}
	\includegraphics[width=\linewidth]{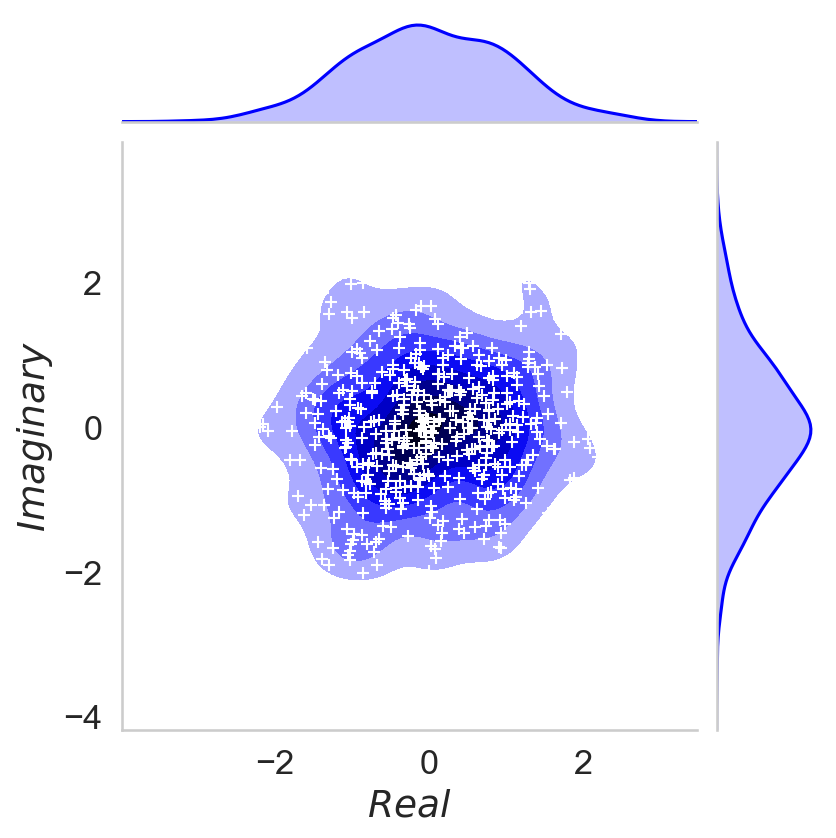}
	\caption{$f(z)=z^{2}$}
\end{subfigure}
	\begin{subfigure}{0.3\textwidth}
	\includegraphics[width=\linewidth]{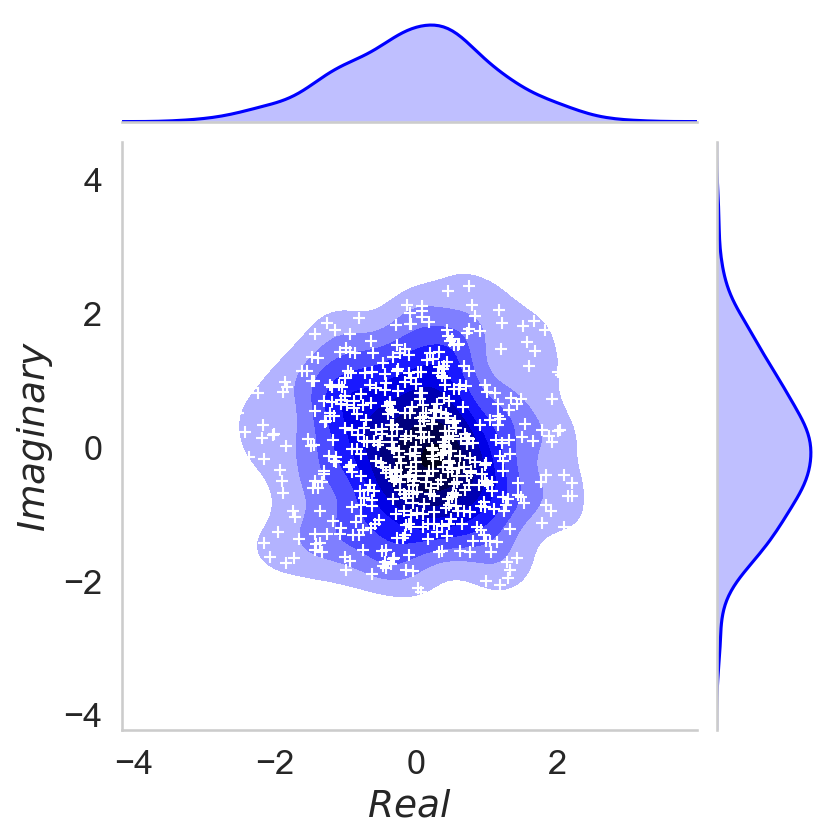}
	\caption{$f(z)=z^{3}$}
\end{subfigure}

	\begin{subfigure}{0.3\textwidth}
		\includegraphics[width=\linewidth]{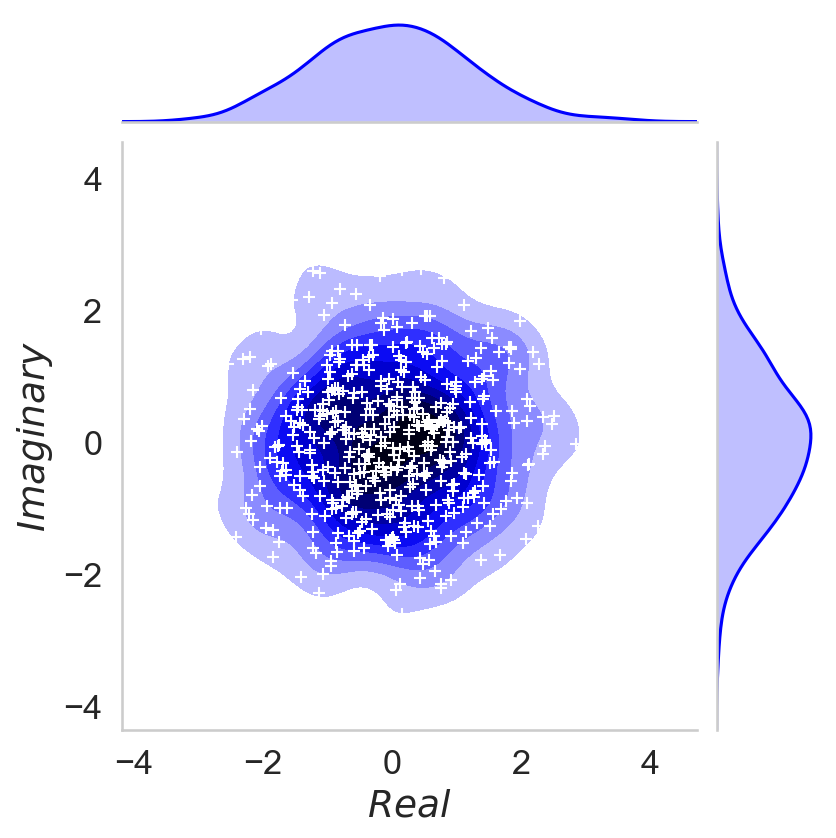}
		\caption{$f(z)=z^{4}$}
	\end{subfigure}
	\begin{subfigure}{0.3\textwidth}
		\includegraphics[width=\linewidth]{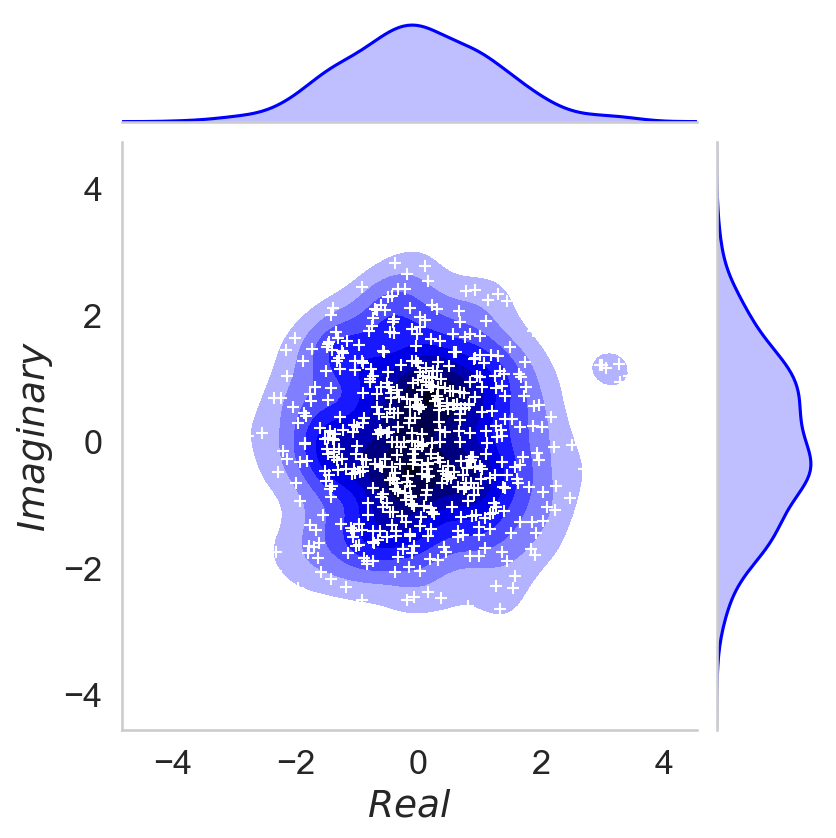}
		\caption{$f(z)=z^{5}$}
	\end{subfigure}
	\begin{subfigure}{0.3\textwidth}
		\includegraphics[width=\linewidth]{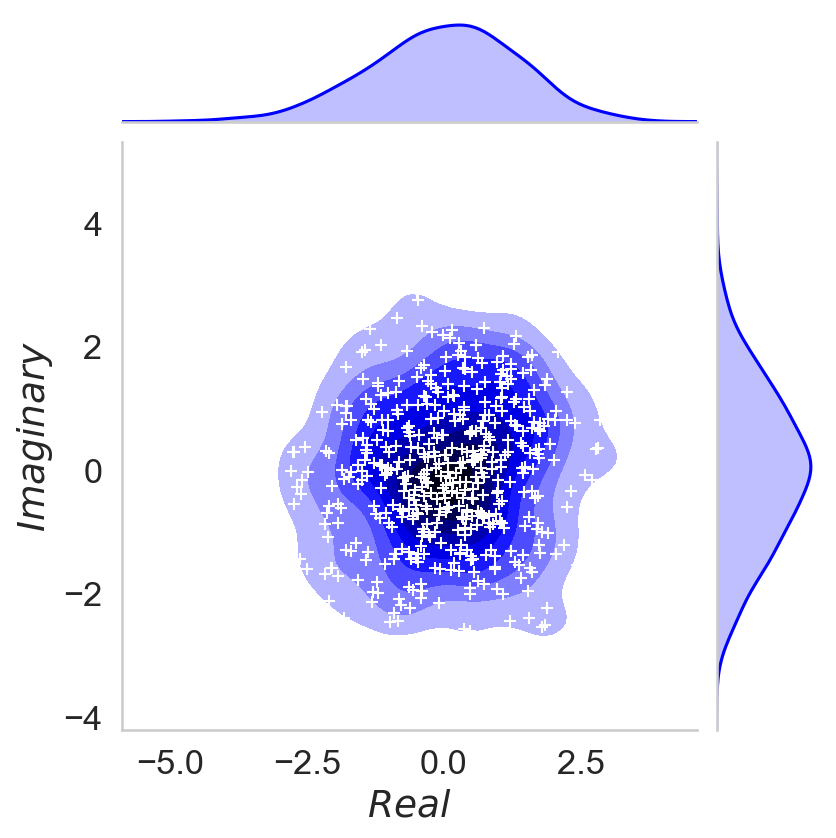}
		\caption{$f(z)=z^{6}$}
	\end{subfigure}

\caption{A heat map of $500$ samples of $\sqrt{c_{n}/n}\CL_{f}^{\Delta}(M)$ and marginal densities of real and imaginary parts are plotted on the complex plane. $M$ is a $4000\times 4000$ random band matrix of bandwidth $b_{n}=n^{0.3}$ with i.i.d. complex Gaussian entries.}\label{fig: band figures}
\end{figure}

\end{enumerate}

\section{Proof of the Theorem \ref{thm: main theorem}}\label{sec: proof of the main theorem}

We adopt the methods based on \cite{rider2006gaussian, shcherbina2015fluctuations}. Let us define the event $$\Omega_{n}:=\{\|M\|\leq \rho\},$$ where $\rho$ is the same as in Lemma \ref{lem: almost sure bound on norm}. From Lemma \ref{lem: almost sure bound on norm} we also have that
\begin{align}\label{eqn: omega complement probablity}
\Pb(\Omega_{n}^{c})\leq K\exp\left(-\sqrt{\frac{\alpha c_{n}}{2\omega}}\frac{3\rho}{4}\right).
\end{align} 
So, if $f:\C\to\C$ is analytic on $\D_{\rho+\tau}$, then on the event $\Omega_{n}$ we may write
\begin{align*}
\CL_{f}(M)&=\frac{1}{2\pi i}\oint_{\partial \D_{\rho+\tau}}\sum_{i=1}^{n}\frac{f(z)}{z-\lambda_{i}(M)}\;dz=\frac{1}{2\pi i}\oint_{\partial \D_{\rho+\tau}}f(z)\tr R_{z}(M)\;dz\\
\text{and}\;\CL_{f}^{\Delta}(M)&=\CL_{f}(M)-nf(0)=\frac{1}{2\pi i}\oint_{\partial\D_{\rho+\tau}}f(z)\tr R_{z}^{\Delta}(M)\;dz,
\end{align*}
where $\tr R_{z}^{\Delta}(M):=\tr R_{z}(M)-\frac{n}{z}$. While the above is true for any such function $f$, the readers may think of $f$ as a linear combination of $f_{1},\ldots,f_{k}$ from theorem \ref{thm: main theorem}. Let us define 
\begin{align*}
\hat{R}_{z}(M):=R_{z}(M)\1_{\Omega_{n}}.
\end{align*}
Now we decompose $\tr R_{z}^{\Delta}(M)$ as
\begin{align*}
\tr R_{z}^{\Delta}(M)=&\tr R_{z}(M)\1_{\Omega_{n}}-\E[\tr R_{z}(M)\1_{\Omega_{n}}]\\
&+\tr R_{z}(M)\1_{\Omega_{n}^{c}}\\
&+\E[\tr R_{z}(M)\1_{\Omega_{n}}]-\frac{n}{z}\\
=&\tr \hat{R}_{z}^{\circ}(M)+\tr R_{z}(M)\1_{\Omega_{n}^{c}}+\E[\tr \hat{R}_{z}(M)]-\frac{n}{z}.
\end{align*}
In the due course, we would like to show that 
\begin{align}
\sup_{z\in \partial\D_{\rho+\tau}}&|\tr R_{z}(M)\1_{\Omega_{n}^{c}}|\stackrel{p}{\to}0\label{eqn: sup of trace of R_z on complement}\\
\sup_{z\in \partial\D_{\rho+\tau}}&\left|\E[\tr \hat{R}_{z}(M)]-\frac{n}{z}\right|\to 0,\label{eqn: difference between trace of R_z and n/z}
\end{align}
and $\hat{R}_{z}^{\circ}(M)$ converges to an appropriate Gaussian process.

\begin{proof}[Proof of \ref{eqn: sup of trace of R_z on complement}]

First of all since the entries of $M$ are continuous random variables, $\Pb(\|M\|=\rho+\tau)=0$. Which implies that $\Pb\left(\sup_{z\in \partial\D_{\rho+\tau}}\left|\tr R_{z}(M)\right|<\infty\right)=1$. Therefore

\begin{align*}
\Pb\left(\sup_{z\in \partial\D_{\rho+\tau}}|\tr R_{z}(M)\1_{\Omega_{n}^{c}}|>\epsilon\right)\leq \Pb(\Omega_{n}^{c})\to 0,\;\text{as}\;n\to\infty.
\end{align*}

\end{proof}

\begin{proof}[Proof of \ref{eqn: difference between trace of R_z and n/z}]
	
To prove \eqref{eqn: difference between trace of R_z and n/z}, we expand $\tr R_{z}(M)$ on $\Omega_{n}$ for $z\in \partial\D_{\rho+\tau}$ as follows
\begin{align}\label{eqn: trace of resolvent expansion}
\tr R_{z}(M)&=\frac{n}{z}+\sum_{l=1}^{\lfloor \log^{2}n\rfloor-1}\frac{1}{z^{l+1}}\tr M^{l}+\frac{1}{z^{\lfloor \ln^{2} n\rfloor+1}}\tr[R_{z}(M)M^{\lfloor \log^{2}n\rfloor}].
\end{align} 
On $\Omega_{n}$, the last term is bounded by $n\tau^{-1}(\rho/|z|)^{\lfloor \log^{2}n\rfloor}$ for all $z\in \partial\D_{\rho+\tau}$.

Now from the condition \ref{con: The condition}(ii, iii) and boundedness of $w(x)$, we shall show that for any $l\geq 3$

\begin{align}\label{eqn: expectation of trace of M^l are zero}
|\E[\tr M^{l}]|&\leq Cn^{-l/6}l^{3l}\nu^{-l/2}=C\left((l^{18}n^{-1}\nu^{-3})^{l/6}\right),
\end{align}
where $C$ is a universal constant.

Under \ref{con: The condition}(iii)(a), it follows that $\E[\tr M^{l}]=\sum \E[m_{i_{1}i_{2}}m_{i_{2}i_{3}}\cdots m_{i_{l}i_{1}}]=0$. Otherwise under condition \ref{con: The condition}(iii)(b), $\E[m_{ij}]=0=\E[m_{ij}^{2}]$. Therefore non-trivial contributions in $\E[\tr M^{l}]$ (in terms of $n$) are obtained when each random variable appears at least three times. There are at most $l^{l}$ different ways to partition $l$ in which each partition size at least three. In each case, $|\{i_{1},i_{2},\ldots,i_{l}\}|\leq n^{l/3}$, and the expectation is bounded by $Cl^{2l}$ (by condition \ref{con: The condition}(ii)). On the other hand, due to normalization by $c_{n}^{-1/2}$, the denominator is $c_{n}^{l/2}=\Omega((\nu n)^{l/2})$. Combining the two, we have the result \eqref{eqn: expectation of trace of M^l are zero}.

Using the equation \eqref{eqn: bound on expectation of norm} along with H\"older's inequality, we have
\begin{align}
&\E[|\tr M^{l}|\1_{\Omega_{n}^{c}}]\leq n\E[\|M\|^{2l}]^{1/2}[\Pb(\Omega_{n}^{c})]^{1/2}\nonumber\\
&\leq n\left[\rho^{l}+\sqrt{K\Gamma(l+1)}\left(\frac{\alpha c_{n}}{2\omega}\right)^{-l/2}\exp\left(\sqrt{\frac{\alpha c_{n}}{2\omega}}\frac{\rho}{8}\right)\right]K\exp\left(-\sqrt{\frac{\alpha c_{n}}{2\omega}}\frac{3\rho}{8}\right)\nonumber\\
&= \left[\rho^{l}\exp\left(-\sqrt{\frac{\alpha c_{n}}{2\omega}}\frac{\rho}{8}\right)+\left(\frac{2\omega (l!)^{1/l}}{\alpha c_{n}}\right)^{l/2}\right]Kn\exp\left(-\sqrt{\frac{\alpha c_{n}}{2\omega}}\frac{\rho}{4}\right)\nonumber\\
&\leq C\left[\rho^{l}+\left(\frac{2\omega (l!)^{1/l}}{\alpha c_{n}}\right)^{l/2}\right]n^{-2}\nonumber\\
&\leq C\left(\rho^{l}+2\right)n^{-2},\label{eqn: expectation of tr M^l on the complement event}
\end{align}
for $l\leq \log^{2}n$ and large enough $n$. Here $C$ is a universal constant.

 In addition, $\E[\tr M]=0=\E[\tr M^{2}]$. As a result taking expectation in \eqref{eqn: trace of resolvent expansion} after multiplying by $\1_{\Omega_{n}}$, and using \eqref{eqn: expectation of tr M^l on the complement event}, \eqref{eqn: expectation of trace of M^l are zero} we have
\begin{align*}
\left|\E[\tr \hat R_{z}(M)]-\frac{n}{z}\right|\leq C_{\tau,\rho}n^{-1/2},
\end{align*} 
where for all $z\in \partial\D_{\rho+\tau}$, the constant $C_{\tau, \rho}$ depends uniformly on $\tau$ and $\rho$. This completes the proof of \eqref{eqn: difference between trace of R_z and n/z}.

\end{proof}

Therefore, asymptotically we may write
%
\begin{align}\label{eqn: Lf(M) interms of the resolvent}
\CL_{f}^{\Delta}(M)&=\frac{1}{2\pi i}\oint_{\partial \D_{\rho+\tau}} f(z)\tr R_{z}^{\Delta}(M)\;dz
=\frac{1}{2\pi i}\oint_{\partial \D_{\rho+\tau}} f(z)\tr \hat{R}_{z}^{\circ}(M)\;dz+o(1).
\end{align}

Thus, proving the Theorem \ref{thm: main theorem} is equivalent to proving the following proposition.

\begin{prop}\label{prop: statement about convergence of resolvent}
The sequence $\{\sqrt{c_{n}/n}\tr \hat{R}_{z}^{\circ}(M)\}_{n}$ is tight in the space of continuous functions on $\partial\D_{\rho+\tau}$, and converges in distribution to a Gaussian process with covariance kernel

\begin{align}\label{stat: statement about convergence of resolvent}
\sigma(z,\eta)=\left\{\begin{array}{ll}\nu\sum_{k\in \Z}\frac{\hat{w}_{\nu}(k)}{(z\bar{\eta}-\hat{w}_{\nu}(k))^{2}}& \\
\text{if $\nu=\lim_{n\to\infty}(c_{n}/n)\in (0,1]$ and $M$ is periodic}& \\\\
\int_{\R}\frac{\hat{w}_{0}(t)}{(z\bar{\eta}-\hat{w}_{0}(t))^{2}}\;dt& \\
\text{if $\lim_{n\to\infty}(c_{n}/n)=0$.}&
\end{array}
\right.
\end{align}

\end{prop}

\begin{rem}\label{rem: eventually it is 1}
Note that since $w$ is continuous and $\int w(x)\;dx=1$, we have $|\hat{w}_{\nu}(k)|, |\hat{w}_{0}(t)|\leq 1$ for all $k\in \Z$, $t\in \R$. Here $\hat{w}$ stands for the Fourier transform of $w$. Therefore $\sigma(z,\eta)$ is well defined for $z,\eta\in \partial\D_{1}$ and analytic on $\D_{1+\epsilon}^{c}$ for any $\epsilon>0$. On the other hand, $f_{i}$s are analytic. Therefore the value of $\oint_{\partial \D_{1+\epsilon}}\oint_{\partial\D_{1+\epsilon}}f_{i}(z)\overline{f_{j}(\eta)}\sigma(z,\eta)\;dz\;d\bar{\eta}$ is same for any $\epsilon>0$. This justifies the integrals in the Theorem \ref{thm: main theorem} are over  $\partial\D_{1}$ instead of  $\partial\D_{\rho+\tau}$.

\end{rem}

Now, we move to the proof of Proposition \ref{prop: statement about convergence of resolvent}. By Cram\'er-Wold device, it suffices to show that for any $\{z_{i},:1\leq i\leq q\}\subset \partial\D_{\rho+\tau}$ and $\{\theta_{i},\beta_{i}:1\leq i\leq q\}\subset \C$ for which 
\begin{align}\label{eqn: Cramer-Wold reduction}
\sqrt{\frac{c_{n}}{n}}\sum_{i=1}^{q}\left[\theta_{i}\tr \hat R_{z_{i}}^{\circ}(M)+\beta_{i}\overline{\tr \hat R_{z_{i}}^{\circ}(M)}\right]\1_{\Omega_{n}}
\end{align}
is real; converges to a Gaussian random variable with mean zero and variance $\sum_{i,j}\theta_{i}\beta_{j}\sigma(z_{i},z_{j})$. In this 
Section, we shall show that it converges to a Gaussian process with mean zero and unknown variance. The exact computation of variance is done in Section \ref{sec: variance calculation}.

We use the martingale difference technique from Lemma \ref{lem: martingale difference CLT} to establish the above. Let $\CE_{k}$ be the averaging with respect to the $k$th column of $M$, and $\E_{k}[\cdot]=\CE_{k+1}\CE_{k+2}\ldots\CE_{n}[\cdot]=\CE_{k+1}\E_{k+1}[\cdot]$. Clearly, $\E_{0}[\cdot]=\E[\cdot]$ and $\E_{n}[\cdot]=[\cdot]$. We write $\tr \hat R_{z}^{\circ}(M)$ as

\begin{align}\label{eqn: trace as martingale difference}
\tr \hat R_{z}^{\circ}(M)=&\sum_{k=1}^{n}(\E_{k}-\E_{k-1})\tr \hat R_{z}(M)
=\sum_{k=1}^{n}\E_{k}\left\{(\tr \hat R_{z}(M))_{k}^{\circ}\right\}
=:\sum_{k=1}^{n}S_{z,k}(M),
\end{align}
where $\xi_{k}^{\circ}=\xi-\CE_{k}[\xi]$. Clearly, $\{S_{z,k}(M)\}_{1\leq k\leq n}$ is a martingale difference sequence with respect to the filtration $\CF_{n,k}=\sigma\{m_{j}:0\leq j\leq k\}$.  Rewrite \eqref{eqn: Cramer-Wold reduction} as

\begin{align*}
\sum_{i=1}^{q}\left[\theta_{i}\tr \hat R_{z_{i}}^{\circ}(M)+\beta_{i}\overline{\tr \hat R_{z_{i}}^{\circ}(M)}\right]=\sum_{k=1}^{n}\left\{\sum_{i=1}^{q}\theta_{i}S_{z_{i},k}(M)+\beta_{i}S_{\bar{z}_{i},k}(M^{*})\right\}=:\sum_{k=1}^{n}\xi_{n,k}.
\end{align*}
Clearly, $\{\xi_{n,k}\}_{1\leq k\leq n}$ is a martingale difference sequence with respect to the filtration $\CF_{n,k}$. Notice that $\E[\cdot|\CF_{n,k-1}]=\E_{k-1}[\cdot]$. Therefore, condition $(i)$ of Lemma \ref{lem: martingale difference CLT} is equivalent to

\begin{align}\label{eqn: equivalence of condition i of martingale clt}
\frac{c_{n}}{n}\sum_{k=1}^{n}\E[\xi_{n,k}^{2}\1_{|\xi_{n,k}|>\delta}]\leq &\frac{c_{n}}{n\delta^{2}}\sum_{k=1}^{n}\E[\xi_{n,k}^{4}]\nonumber\\
\leq &\frac{(2q)^{3}c_{n}}{n\delta^{2}}\sum_{k=1}^{n}\sum_{i=1}^{q}(|\theta_{i}|^{4}+|\beta_{i}|^{4})\E\left[|S_{z,k}(M)|^{4}\right]\to 0.
\end{align}
And the condition $(ii)$ of Lemma \ref{lem: martingale difference CLT} is equivalent to
\begin{align}
&\frac{c_{n}}{n}\sum_{k=1}^{n}\E_{k-1}\left[S_{z_{i},k}(M)S_{z_{j},k}(M)\right]\stackrel{p}{\to}0,\label{eqn: (requirement) sum of product of matringale difference converges to zero}\\
\text{and}\;\;\;&\frac{c_{n}}{n}\sum_{k=1}^{n}\E_{k-1}\left[S_{z_{i},k}(M)S_{\bar{z}_{j},k}(M^{*})\right]\stackrel{p}{\to}\sigma(z_{i},z_{j})\label{eqn: (requirement) sum of conjugate product of matringale difference converges to the correct variance}.
\end{align}
Thus proving the Proposition \ref{prop: statement about convergence of resolvent} is equivalent to showing the tightness of $\{\sqrt{c_{n}/n}\tr \hat{R}_{z}^{\circ}(M)\}_{n}$ and \eqref{eqn: equivalence of condition i of martingale clt}, \eqref{eqn: (requirement) sum of product of matringale difference converges to zero}, \eqref{eqn: (requirement) sum of conjugate product of matringale difference converges to the correct variance}. We do the following general reductions before the proof.

Let $M^{(k)}$ be obtained by setting the $k$th column of $M$ to zero, and $\hat{R}_{z}(M^{(k)}):=R_{z}(M^{(k)})\1_{\Omega_{n,k}}$, where $\Omega_{n,k}=\{\|M^{(k)}\|\leq \rho\}$. The resolvent identity yields $\tr R_{z}(M)=\tr R_{z}(M^{(k)})+e_{k}^{t}R_{z}(M^{(k)})R_{z}(M)m_{k}$. We show that replacing $R_{z}(M^{(k)})\1_{\Omega_{n}}$ by $R_{z}(M^{(k)})\1_{\Omega_{n,k}}$ in \eqref{eqn: trace as martingale difference} is asymptotically same in probability.
\begin{align*}
&\E\left|\sum_{k=1}^{n}\E_{k}\left[\tr R_{z}(M^{(k)})\1_{\Omega_{n}}-\tr R_{z}(M^{(k)})\1_{\Omega_{n,k}}\right]_{k}^{\circ}\right|\\
\leq&2\E\left|\sum_{k=1}^{n}\E_{k}\left[\tr R_{z}(M^{(k)})\1_{\Omega_{n}^{c}\cap\Omega_{n,k}}\right]\right|\\
\leq&2n^{2}\tau^{-1}\Pb(\Omega_{n}^{c})\to 0,\;\;\text{as $n\to\infty$}.
\end{align*} 
Therefore using resolvent identity and Lemma \ref{lem: Sherman-Morrison formula}, we have 
\begin{align*}
\tr \hat R_{z}(M)=&\tr \hat R_{z}(M^{(k)})+e_{k}^{t}\hat R_{z}(M^{(k)})\hat R_{z}(M)m_{k}\\
=&\tr \hat R_{z}(M^{(k)})+\frac{e_{k}^{t}\hat R_{z}(M^{(k)})^{2}m_{k}}{1+e_{k}^{t}\hat R_{z}(M^{(k)})m_{k}}\\
=&\tr \hat R_{z}(M^{(k)})-\frac{\partial}{\partial z}\log \{1+\delta_{k}(z)\},
\end{align*}
where
\begin{align*}
\delta_{k}(z)&:=e_{k}^{t}\hat R_{z}(M^{(k)})m_{k}.
\end{align*}

We notice that $\delta_{k}(z)$ is a product of two independent random variables. Intuitively, conditioned on $\{m_{j}:j\in\{1,2,\ldots, n\}\bs k\}$, $\delta_{k}(z)\to 0$ almost surely. We give the exact estimate below. 

Define $h(m_{k})=e_{k}^{t}\hat R_{z}(M^{(k)})m_{k}=\delta_{k}(z)$. Then using the property \ref{Poincare exponential tail bound} in Definition \ref{defn: poincare inequality} and the fact that $\sum_{j=1}^{n}|\hat R_{z}(M^{(k)})_{ij}|^{2}\leq \|\hat R_{z}(M^{(k)})\|^{2}\leq \tau^{-2}$ for $z\in \partial\D_{\rho+\tau}$, we have

\begin{align}\label{eqn: Estimates of moments of delta}
\Pb\left(|\delta_{k}(z)|>t\;|\;\hat R_{z}(M^{(k)})\right)\leq K\exp\left(-\tau\sqrt{\frac{\alpha c_{n}}{2\omega}}t\right),
\end{align}

where $\omega=\sup_{x}w(x)$. The factor $c_{n}/\omega$ is obtained by using Definition \ref{defn: poincare inequality}(1) and the fact that each entry of $M$ is scaled by the variance profile $w$ and $1/\sqrt{c_{n}}$. Now, taking $t_{n}=c_{n}^{-1/8}$ and using the fact that $c_{n}\geq \log^{3}n$ we have $\sum_{n}\exp\left(-\tau\sqrt{\frac{\alpha c_{n}}{2\omega}}t_{n}\right)<\infty$. Hence by Borel–Cantelli lemma, we have $\delta_{k}(z)\stackrel{a.s.}{\to} 0$. In addition, we also get that
\begin{align*}
\E[|\delta_{k}(z)|^{p}]\leq K p!\left(\frac{2\omega}{\tau^{2}\alpha c_{n}}\right)^{p/2}\;\;\forall\;z\in \partial\D_{\rho+\tau}.
\end{align*}

We notice that 
\begin{align}\label{eqn: M^k is already centered}
[\tr \hat R_{z}(M^{(k)})]_{k}^{\circ}=\tr \hat R_{z}(M^{(k)})-\CE_{k}[\tr \hat R_{z}(M^{(k)})]=0.
\end{align}
Because $M^{(k)}$ is independent of the $k$th column $m_{k}$. 
Consequently,

\begin{align*}
S_{z,k}(M)=& \E_{k}\left\{(\tr R_{z}(M))_{k}^{\circ}\right\}
=-\E_{k}\left[\frac{\partial}{\partial z}(\log\{1+\delta_{k}(z)\})_{k}^{\circ}\right].
\end{align*}

Now we show that the above expectation exists. Let $p_{\delta_{k}}$ be the pdf of $1+\delta_{k}(z)|\hat R_{z}(M^{(k)})$. Since the pdf of each $x_{ij}$ are bounded (condition \ref{con: The condition}(ii)), by Young's convolution inequality, $\|p_{\delta_{k}}\|_{\infty}\leq C\tau\sqrt{c_{n}}$. So  for any $1\leq s<2$ and large enough $n$,
\begin{align*}
\E[|1+\delta_{k}(z)|^{-s}|\hat R_{z}(M^{(k)})]&=\int_{\D_{1/2}}r^{-s}p_{\delta_{k}}(r,\theta)\;r\;dr\;d\theta+\int_{\D_{1/2}^{c}}r^{-s}p_{\delta_{k}}(r,\theta)\;r\;dr\;d\theta\\
&\leq  C\frac{2^{s-2}}{2-s}\tau\sqrt{c_{n}}+K2^{s-1}\\
&\leq C'\frac{2^{s-2}}{2-s}\tau\sqrt{c_{n}},
\end{align*}
where $C'$ is a universal constant. Taking further expectation,
\begin{align}\label{eqn: expectation of (1+delta)^-1 is bounded}
\E[|1+\delta_{k}(z)|^{-s}]\leq C'\frac{2^{s-2}}{2-s}\tau\sqrt{c_{n}},\;\;\forall\;1\leq s<2.
\end{align}
Similarly also using \eqref{eqn: Estimates of moments of delta}, $\E[|\log(1+\delta_{k}(z))|]\leq C'\tau\sqrt{c_{n}}$. Therefore, we have

\begin{align}\label{eqn: Sz interms of log differentiation}
S_{z,k}(M)=& \E_{k}\left\{(\tr \hat R_{z}(M))_{k}^{\circ}\right\}
=-\frac{\partial}{\partial z}\E_{k}[(\log\{1+\delta_{k}(z)\})_{k}^{\circ}],
\end{align}
where the switching between $\E_{k}$ and $\partial/\partial z$ can be justified by the dominated convergence theorem and \eqref{eqn: expectation of (1+delta)^-1 is bounded}. The above technique rewrites $S_{z,k}(M)$ in terms of derivative of an analytic function, which will allow us to estimate $S_{z,k}(M)$ via the following basic fact from complex analysis;

\begin{align}\label{eqn: passing estimates of g to its derivative}
|g'(z)|=\left|\frac{1}{2\pi i}\oint_{z+\partial\D_{\tau/3}}\frac{g(\zeta)\;d\zeta}{(\zeta-z)^{2}}\right|\leq \frac{5}{\pi \tau^{2}}\oint_{z+\partial\D_{\tau/3}}|g(\zeta)|\;d\zeta,
\end{align}

where $g$ is analytic on $z+\D_{\tau/2}$.

\subsection{Proof of tightness} This part is similar to Section 3.2 in \cite{rider2006gaussian}. In this subsection, we show that $\{\sqrt{c_{n}/n}\tr \hat{R}_{z}^{\circ}(M)\}_{n}$ is tight in  $\mathcal{C}(\partial \D_{\rho+\tau})$, the space of continuous functions on $\partial\D_{\rho+\tau}$. In other words, for any $\epsilon>0$ there exists a compact set $\mathcal{K}(\epsilon)\subset \mathcal{C}(\partial \D_{\rho+\tau})$ such that 
\begin{align}\label{eqn: tightness definition}
\Pb\left(\sqrt{\frac{c_{n}}{n}}\tr \hat{R}_{z}^{\circ}(M)\in \mathcal{K}(\epsilon)^{c}\right)<\epsilon.
\end{align}
However, the compact sets in the space of continuous functions are the space of equicontinuous functions. 

Let us define,
\begin{align*}
\Xi_{n}:=\{|\delta_{k}(z)|<c_{n}^{-1/8},\;\;\forall\;1\leq k\leq n\}.
\end{align*}
By \eqref{eqn: Estimates of moments of delta}, and simple union bound
\begin{align*}
\Pb(\Xi_{n}^{c})\leq K\exp\left(\log n-\tau\sqrt{\frac{\alpha}{2\omega}}c_{n}^{3/8}\right)=o(1).
\end{align*}

The last equality follows from the assumption that $c_{n}\geq \log^{3}n$. Consequently,
\begin{align*}
\Pb\left(\sqrt{\frac{c_{n}}{n}}\tr \hat{R}_{z}^{\circ}(M)\in \mathcal{K}(\epsilon)^{c}\right)
&\leq \Pb\left(\sqrt{\frac{c_{n}}{n}}\left|\frac{\tr \hat{R}_{z}^{\circ}(M)-\tr \hat{R}_{z}^{\circ}(M)}{z-\eta}\right|>K\right)\\
&\leq \Pb\left(\sqrt{\frac{c_{n}}{n}}\left|\frac{\tr \hat{R}_{z}^{\circ}(M)-\tr \hat{R}_{z}^{\circ}(M)}{z-\eta}\right|\1_{\Xi_{n}}>K\right)+o(1),
\end{align*}
for any $K>0$. Now, \eqref{eqn: tightness definition} follows from Markov's inequality and the following condition;
\begin{align}\label{eqn: equivalent tightness condition}
\E\left\{\frac{c_{n}}{n}\left|\frac{\tr \hat{R}_{z}^{\circ}(M)-\tr \hat{R}_{z}^{\circ}(M)}{z-\eta}\right|^{2}\1_{\Xi_{n}}\right\}\leq C,
\end{align}
uniformly for all $z,\eta\in \partial\D_{\rho+\tau}$ and $n\in \N$. In what follows, the methods are similar to \cite{rider2006gaussian}. For the sake of completeness, we outline it here.

Using the resolvent identity, we can write
\begin{align*}
\tr \hat R_{z}(M)-\tr \hat R_{\eta}(M)=(\eta -z)\tr [\hat R_{z}(M)\hat R_{\eta}(M)].
\end{align*}
Thus, in the view of \eqref{eqn: trace as martingale difference}
\begin{align}\label{eqn: tightness difference between traces}
(\eta-z)^{-1}\left[\tr \hat R_{z}^{\circ}(M)-\tr \hat R_{\eta}^{\circ}(M)\right]=\sum_{k=1}^{n}\E_{k}\{(\tr [\hat R_{z}(M)\hat R_{\eta}(M)])_{k}^{\circ}\}.
\end{align}

Using resolvent identity and Lemma \ref{lem: Sherman-Morrison formula},
\begin{align*}
&\tr [\hat R_{z}(M)\hat R_{\eta}(M)-\hat R_{z}(M^{(k)})\hat R_{\eta}(M^{(k)})]\\
=&\tr[\{\hat R_{z}(M)-\hat R_{z}(M^{(k)})\}\{\hat R_{\eta}(M)-\hat R_{\eta}(M^{(k)})\}]\\
&+\tr [\hat R_{z}(M^{(k)})\{\hat R_{\eta}(M)-\hat R_{\eta}(M^{(k)})\}]+\tr[\{\hat R_{z}(M)-\hat R_{z}(M^{(k)})\}\hat R_{\eta}(M^{(k)})]\\
=&\tr \frac{\hat R_{z}(M^{(k)})m_{k}e_{k}^{t}\hat R_{z}(M^{(k)})}{1+\delta_{k}(z)}\frac{\hat R_{\eta}(M^{(k)})m_{k}e_{k}^{t}\hat R_{\eta}(M^{(k)})}{1+\delta_{k}(\eta)}\\
&+\tr \frac{\hat R_{z}(M^{(k)})\hat R_{\eta}(M^{(k)})m_{k}e_{k}^{t}\hat R_{\eta}(M^{(k)})}{1+\delta_{k}(\eta)}+\tr \frac{\hat R_{z}(M^{(k)})m_{k}e_{k}^{t}\hat R_{z}(M^{(k)})\hat R_{\eta}(M^{(k)})}{1+\delta_{k}(z)}\\
=&\frac{\left\{e_{k}^{t}\hat R_{z}(M^{(k)})\hat R_{\eta}(M^{(k)})m_{k}\right\}^{2}}{(1+\delta_{k}(z))(1+\delta_{k}(\eta))}+\frac{e_{k}^{t}\hat R_{\eta}(M^{(k)})\hat R_{z}(M^{(k)})\hat R_{\eta}(M^{(k)})m_{k}}{1+\delta_{k}(\eta)}\\
+&\frac{e_{k}^{t}\hat R_{z}(M^{(k)})\hat R_{\eta}(M^{(k)})\hat R_{z}(M^{(k)})m_{k}}{1+\delta_{k}(z)}\\
=:&\Theta_{1}(k)+\Theta_{2}(k)+\Theta_{3}(k).
\end{align*}
In the view of \eqref{eqn: M^k is already centered}, we have $[\tr \hat R_{z}(M^{(k)})\hat R_{\eta}(M^{(k)})]_{k}^{\circ}=0$. Therefore we can rewrite \eqref{eqn: tightness difference between traces} as
\begin{align*}
(\eta-z)^{-1}\left[\tr \hat R_{z}^{\circ}(M)-\tr \hat R_{\eta}^{\circ}(M)\right]=\sum_{k=1}^{n}\E_{k}\left\{\left[\Theta_{1}(k)+\Theta_{2}(k)+\Theta_{3}(k)\right]_{k}^{\circ}\right\}.
\end{align*}
Since $\{[\Theta_{i}(k)]_{k}^{\circ}\}$ is a martingale difference sequence,
\begin{align*}
\E\left|\sum_{k=1}^{n}\E_{k}\left\{\left[\Theta_{1}(k)+\Theta_{2}(k)+\Theta_{3}(k)\right]_{k}^{\circ}\right\}\right|^{2}=\sum_{k=1}^{n}\E\left|\E_{k}\left\{\left[\Theta_{1}(k)+\Theta_{2}(k)+\Theta_{3}(k)\right]_{k}^{\circ}\right\}\right|^{2}.
\end{align*}
Therefore, proving \eqref{eqn: equivalent tightness condition} is equivalent to showing that 
\begin{align}\label{eqn: Theta_i expectations}
c_{n}\E[|\Theta_{i}(k)|^{2}\1_{\Xi_{n}}]\leq C',\;\;\forall\;1\leq k\leq n, \;i=1,2,3,
\end{align}
uniformly for all $n\in \N$ and $z,\eta\in \partial\D_{\rho+\tau}$. However, applying the same method as described in \eqref{eqn: Estimates of moments of delta}, we can get similar tail estimates for $e_{k}^{t}\hat R_{z}(M^{(k)})\hat R_{\eta}(M^{(k)})m_{k}$ etc. (with $\tau^{2}$ or $\tau^{3}$ in the rhs of \eqref{eqn: Estimates of moments of delta}). Now since $|\delta_{k}(z)|<c_{n}^{-1/8}$ on $\Xi_{n}$, using the estimate $|(1+\delta_{k}(z))|^{-1}\leq 2$ in $\E[|\Theta_{i}(k)|^{2}]$s, we have \eqref{eqn: Theta_i expectations}.

\subsection{Proof of \eqref{eqn: equivalence of condition i of martingale clt}} Expanding $\log\{1+\delta_{k}(z)\}$ up to two terms and using \eqref{eqn: Estimates of moments of delta}, \eqref{eqn: Sz interms of log differentiation}, \eqref{eqn: passing estimates of g to its derivative} we have 
\begin{align*}
\E[|S_{z,k}(M)|^{4}]=O(c_{n}^{-2}).
\end{align*}
Substituting the above in \eqref{eqn: equivalence of condition i of martingale clt}, we obtain
\begin{align*}
\sum_{k=1}^{n}\E[\xi_{n,k}^{2}\1_{|\xi_{n,k}|>\delta}]=\delta^{-2}O(c_{n}^{-1}),
\end{align*}
which proves the result.

\subsection{Proof of \eqref{eqn: (requirement) sum of product of matringale difference converges to zero}}
Using condition \ref{con: The condition}(iii), \eqref{eqn: Estimates of moments of delta} and expanding $\log\left\{1+\delta_{k}(z)\right\}$ up to two terms, we see that 
%
\begin{align*}
\CE_{k}\left\{\E_{k}[(\log\{1+\delta_{k}(z)\})_{k}^{\circ}]\E_{k}[(\log\{1+\delta_{k}(\eta)\})_{k}^{\circ}] \right\}=O(c_{n}^{-2}).
\end{align*}
Thus, using \eqref{eqn: Sz interms of log differentiation}, \eqref{eqn: passing estimates of g to its derivative} and the above we have
\begin{align*}
\frac{c_{n}}{n}\sum_{k=1}^{n}\E_{k-1}\left[S_{z_{i},k}(M)S_{z_{j},k}(M)\right]=O(c_{n}^{-1}),
\end{align*}
which proves \eqref{eqn: (requirement) sum of product of matringale difference converges to zero}.

\subsection{Proof of \eqref{eqn: (requirement) sum of conjugate product of matringale difference converges to the correct variance}}
Expanding $\log\{1+\delta_{k}(z)\}$ up to two terms and using \eqref{eqn: Estimates of moments of delta}, \eqref{eqn: Sz interms of log differentiation} we have
\begin{align*}
\E_{k-1}[S_{z,k}(M)S_{\bar{\eta},k}(M^{*})]=&\frac{\partial^{2}}{\partial z\partial\bar{\eta}}D_{k}(z,\bar{\eta}),
\end{align*}
where
\begin{align*}
D_{k}(z,\bar{\eta})
=&\CE_{k}\left\{\E_{k}[(\log\{1+\delta_{k}(z)\})_{k}^{\circ}]\E_{k}[(\log\{1+\overline{\delta_{k}(\eta)}\})_{k}^{\circ}]\right\}\\
=&\frac{1}{c_{n}}e_{k}^{t}\E_{k}[R_{z}(M^{(k)})]\CI_{k}\E_{k}[R_{\bar{\eta}}(M^{(k*)})e_{k}]+O(c_{n}^{-2})\\
=:&\frac{1}{c_{n}}T_{k}(z,\bar{\eta})+O(c_{n}^{-2}).
\end{align*}

As a result, \eqref{eqn: (requirement) sum of conjugate product of matringale difference converges to the correct variance} becomes
\begin{align*}
\frac{c_{n}}{n}\sum_{k=1}^{n}\E_{k-1}[S_{z,k}(M)S_{\bar{\eta},k}(M^{*})]
=&\frac{\partial^{2}}{\partial z\partial\bar{\eta}}\left[\frac{1}{n}\sum_{k=1}^{n}T_{k}(z,\bar{\eta})\right]+O(c_{n}^{-1})\\
=:&\frac{\partial^{2}}{\partial z\partial\bar{\eta}}U(z,\bar{\eta})+O(c_{n}^{-1}).
\end{align*}

Now since $\|\hat{R}_{z}(M^{(k)})\|\leq 2\tau^{-1}$ on $\D_{\rho+\tau/2}^{c}$, $U(z,\bar{\eta})$ is a sequence of uniformly bounded analytic functions on $\D_{\rho+\tau/2}^{c}$. Therefore by Vitali's theorem (eg. \cite[Theorem 5.21]{titchmarsh1939theory}), proving \eqref{eqn: (requirement) sum of conjugate product of matringale difference converges to the correct variance} equivalent to show that $U(z,\bar{\eta})$ converges in probability. 


 Since $\{S_{z,k}(M)\}_{k}$ is a martingale difference sequence, we have $\E[S_{z,k}(M)S_{\bar{\eta},l}(M^{*})]=0$ if $k\neq l$. As a result, 
\begin{align*}
\E\left\{\sum_{k=1}^{n}\E_{k-1}\left[S_{z,k}(M)S_{\bar{\eta},k}(M^{*})\right]\right\}
=&\E\left\{\left(\sum_{k=1}^{n}S_{z,k}(M)\right)\left(\sum_{l=1}^{n}S_{\bar{\eta},l}(M^{*})\right)\right\}\\
=&\E[\tr \hat R_{z}^{\circ}(M)\tr \hat R_{\bar{\eta}}^{\circ}(M^{*})].
\end{align*}
Limit of the above is calculated in Section \ref{sec: variance calculation}. Here we show that $\Var(U(z,\bar{\eta}))\to 0$.

Recall
\begin{align*}
T_{k}(z,\bar{\eta})=e_{k}^{t}\E_{k}\left[\hat R_{z}(M^{(k)})\right]\CI_{k}\E_{k}\left[\hat R_{\eta}(M^{(k)})^{*}\right]e_{k}.
\end{align*}
Let $\epsilon_{n}=2/c_{n}$, and $\phi_{\epsilon_{n}, k}:\R^{2n}\to [0,1]$ be a smooth function such that
\begin{align*}
&\left.\phi_{\epsilon_{n}, k}\right|_{\{\|M^{(k)}\|\leq \rho\}}\equiv 1,\\
& \left.\phi_{\epsilon_{n}, k}\right|_{\{\|M^{(k)}\|\geq \rho+\epsilon_{n}\}}\equiv 0,\\
& \left|\frac{\partial \phi_{\epsilon_{n},k}}{\partial x_{i}}\right|\leq \frac{2}{\epsilon_{n}},\;\;\forall 1\leq i\leq 2n, x\in \R^{2n}.
\end{align*}
Let us define $\tilde{R}_{z}(M^{(k)})=R_{z}(M^{(k)})\phi_{\epsilon_{n}, k}$ and
\begin{align*}
\tilde T_{k}(z,\bar{\eta})=e_{k}^{t}\E_{k}\left[\tilde R_{z}(M^{(k)})\right]\CI_{k}\E_{k}\left[\tilde R_{\eta}(M^{(k)})^{*}\right]e_{k}.
\end{align*}
We note the following estimate
\begin{align*}
|T_{k}(z,\bar{\eta})-\tilde{T}_{k}(z,\bar{\eta})|\leq \frac{20}{(\tau-\epsilon_{n})^{2}}\phi_{\epsilon_{n}, k}\1_{\{\rho<\|M^{(k)}\|<\rho+\epsilon_{n}\}}.
\end{align*}
Using the Lemma \ref{lem: almost sure bound on norm}, we have
\begin{align}\label{eqn: difference between T and tilde T}
\E\left[|T_{k}(z,\bar{\eta})-\tilde{T}_{k}(z,\bar{\eta})|^{2}\right]\leq \frac{K}{(\tau-\epsilon_{n})^{4}}\exp\left(-\sqrt{\frac{\alpha c_{n}}{2\omega}}\frac{3\rho}{4}\right).
\end{align}
For notational simplicity, let us denote 
\begin{align*}
u:=\E_{k}\left[R_{z}(M^{(k)})\right]\sqrt{\CI_{k}},\;\;\;
v:=\sqrt{\CI_{k}}\E_{k}\left[R_{\eta}(M^{(k)})^{*}\right],\\
\tilde u:=\E_{k}\left[\tilde R_{z}(M^{(k)})\right]\sqrt{\CI_{k}},\;\;\;
\tilde v:=\sqrt{\CI_{k}}\E_{k}\left[\tilde R_{\eta}(M^{(k)})^{*}\right],
\end{align*}
where $\sqrt{\CI_{k}}$ denotes the diagonal matrix by taking square root of each entry of the diagonal matrix $\CI_{k}$. Then $\tilde T_{k}(z,\bar{\eta})=\sum_{s\in I_{k}}\tilde u_{ks}\tilde v_{sk}$. Since $x_{ij}$s satisfy Poincar\'e inequality, we have
\begin{align*}
\Var(\tilde T_{k}(z,\bar{\eta}))\leq \frac{1}{\alpha}\sum_{j=1}^{k-1}\sum_{i\in I_{j}}\left[\E \left|\frac{\partial \tilde T_{k}(z,\bar{\eta})}{\partial x_{ij}}\right|^{2}+\E \left|\frac{\partial \tilde T_{k}(z,\bar{\eta})}{\partial \bar x_{ij}}\right|^{2}\right].
\end{align*}

The first sum stops at $k-1$ because $\tilde T_{k}(z,\bar{\eta})$ is constant as a function of $k,k+1,\ldots, n$ columns of $M$. On the other hand, we have

\begin{align*}
&\frac{\partial R_{z}(M^{(k)})_{ks}}{\partial m_{ij}}=R_{z}(M^{(k)})_{ki}R_{z}(M^{(k)})_{js},\;
\frac{\partial \overline{R_{z}(M^{(k)})}_{sk}}{\partial m_{ij}}=0.
\end{align*}
Consequently,
\begin{align*}
&\frac{\partial \tilde u_{ks}}{\partial m_{ij}}=\tilde u_{ki}\tilde u_{js}+u_{ks}\frac{\partial \phi_{\epsilon_{n},k}}{\partial m_{ij}}\1_{\{\rho<\|M^{(k)}\|<\rho+\epsilon_{n}\}},\\
&\frac{\partial \tilde v_{sk}}{\partial m_{ij}}=v_{sk}\frac{\partial \phi_{\epsilon_{n},k}}{\partial m_{ij}}\1_{\{\rho<\|M^{(k)}\|<\rho+\epsilon_{n}\}},\\
&\frac{\partial(\tilde T_{k}(z, \bar{\eta}))}{\partial m_{ij}}=\sum_{s\in I_{k}}\tilde u_{ki}\tilde u_{js}\tilde v_{sk}+\sum_{s\in I_{k}}(u_{ks}\tilde{v}_{sk}+\tilde{u}_{ks}v_{sk})\frac{\partial \phi_{\epsilon_{n},k}}{\partial m_{ij}}\1_{\{\rho<\|M^{(k)}\|<\rho+\epsilon_{n}\}}.
\end{align*}

Denoting $\tilde y_{jk}=\sum_{s\in I_{k}}\tilde u_{js}\tilde v_{sk}$ and using the facts that $\|\tilde u\|,\|\tilde v\|\leq (\tau-\epsilon_{n})^{-1}$, we have
\begin{align*}
&\sum_{j=1}^{k-1}\sum_{i\in I_{j}}\left|\sum_{s\in I_{k}}\tilde u_{ki}\tilde u_{js}\tilde v_{sk}\right|^{2}
=\sum_{j=1}^{k-1}\sum_{i\in I_{j}}|\tilde u_{ki}\tilde y_{jk}|^{2}\leq \|\tilde u_{k}\|_{2}^{2}\|\tilde y_{k}\|_{2}^{2}\leq (\tau-\epsilon_{n})^{-6},\\
&\sum_{j=1}^{k-1}\sum_{i\in I_{j}}\left|\sum_{s\in I_{k}}(u_{ks}\tilde{v}_{sk}+\tilde{u}_{ks}v_{sk})\frac{\partial \phi_{\epsilon_{n},k}}{\partial m_{ij}}\1_{\{\rho<\|M^{(k)}\|<\rho+\epsilon_{n}\}}\right|^{2}\\
&\leq \frac{8}{\epsilon_{n}^{2}}\sum_{j=1}^{k-1}\sum_{i\in I_{j}}\|\tilde{u}_{k}\|_{2}^{2}\|\tilde{v}_{k}\|_{2}^{2}\1_{\{\rho<\|M^{(k)}\|<\rho+\epsilon_{n}\}}\leq 8nc_{n}^{3}(\tau-\epsilon_{n})^{-4}\1_{\{\rho<\|M^{(k)}\|<\rho+\epsilon_{n}\}},
\end{align*}
Using the above estimates, \eqref{eqn: omega complement probablity}, and the fact that $\partial m_{ij}/\partial x_{ij}=O(c_{n}^{-1/2})$ we have,

\begin{align*}
\Var\left(\frac{1}{n}\sum_{k=1}^{n}\tilde T_{k}(z,\bar{\eta})\right)\leq \frac{1}{n}\sum_{k=1}^{n}\Var(\tilde T_{k}(z,\bar{\eta}))\leq \frac{2}{\alpha c_{n}(\tau-\epsilon_{n})^{6}}+\frac{Knc_{n}^{2}}{(\tau-\epsilon_{n})^{4}}\exp\left(-\sqrt{\frac{\alpha c_{n}}{2\omega}}\frac{3\rho}{4}\right).
\end{align*}
No using the estimate \eqref{eqn: difference between T and tilde T} and the assumption $c_{n}>\log^{3}n$, we conclude that

\begin{align*}
&\Var(U(z,\bar{\eta}))=\Var\left(\frac{1}{n}\sum_{k=1}^{n}T_{k}(z,\bar{\eta})\right)\\
&\leq \frac{1}{n}\sum_{k=1}^{n}\E\left[|T_{k}(z,\bar{\eta})-\tilde{T}_{k}(z,\bar{\eta})|^{2}\right]+\Var\left(\frac{1}{n}\sum_{k=1}^{n}\tilde T_{k}(z,\bar{\eta})\right)\\
&\leq \frac{2}{\alpha c_{n}(\tau-\epsilon_{n})^{6}}+\frac{Knc_{n}^{2}}{(\tau-\epsilon_{n})^{4}}\exp\left(-\sqrt{\frac{\alpha c_{n}}{2\omega}}\frac{3\rho}{4}\right)\to 0,\;\text{as $n\to\infty$}.
\end{align*}

\section{Calculation of the variance}\label{sec: variance calculation}

Let us first find the variance for monomial test functions. Let us define $f(z)=z^{l};\;\;l\geq 2$.

\subsection{Case I: $\nu=\lim_{n\to\infty}\frac{c_{n}}{n}\in (0,1]$} The matrix in Definition \ref{defn: band matrix with a variance profile} is periodic.

\begin{align*}
&\lim_{n\to\infty}\frac{c_{n}}{n}\E\left[\CL_{f}(M)\overline{\CL_{f}(M)}\right]
=\lim_{n\to\infty}\frac{c_{n}}{n}\E[\tr M^{l}\tr M^{*l}]\\
=&\lim_{n\to\infty}\frac{c_{n}}{n}\E\left[\left\{\sum_{i_{1},\ldots, i_{l}=1}^{n} m_{i_{1}i_{2}}m_{i_{2}i_{3}}\cdots m_{i_{l}i_{1}}\right\}\left\{\sum_{i_{1},\ldots, i_{l}=1}^{n} \overline m_{j_{1}j_{2}}\overline m_{j_{2}j_{3}}\cdots \overline m_{j_{l}j_{1}}\right\}\right].
\end{align*}

In the above expression, the maximum contribution (in terms of $n$) occurs when all the indices in the loop $i_{1}\to i_{2}\to i_{3}\to\cdots\to i_{l}\to i_{1}$ are distinct and the loop overlaps with the loop $j_{1}\to j_{2}\to j_{3}\to\cdots\to j_{l}\to j_{1}$. The reasoning is similar to \eqref{eqn: expectation of trace of M^l are zero}. Once the indices $i_{1}\to i_{2}\to i_{3}\to\cdots\to i_{l}\to i_{1}$ are fixed, the loop $j_{1}\to j_{2}\to j_{3}\to\cdots\to j_{l}\to j_{1}$ must be same as the loop $i_{1}\to i_{2}\to i_{3}\to\cdots\to i_{l}\to i_{1}$. However, they can overlap in $l$ different ways by rotating $i_{1}\to i_{2}\to i_{3}\to\cdots\to i_{l}\to i_{1}$. 

Now the first index $i_{1}$ can be chosen in $n$ different ways. After that, while choosing the remaining $(l-1)$ many indices, due to the band matrix structure, each index has to be within $\pm b_{n}$ neighborhood of the previous index such that the final index $i_{l}$ is also within $\pm b_{n}$ neighborhood of the first index $i_{1}$ as well.

This last condition imposes an additional constraint which is not present in the full matrix cases. However, these band constraints are completely handled by the weight profile $w_{\nu}$. Therefore using the structures of $m_{ij}$s from Definition \ref{defn: band matrix with a variance profile}, we have

\begin{align}
&\lim_{n\to\infty}\frac{c_{n}}{n}\E\left[\CL_{f}(M)\overline{\CL_{f}(M)}\right]\nonumber\\
=&\lim_{n\to\infty}\frac{c_{n}}{n}\cdot\frac{ln}{c_{n}^{l}}\sum_{1\leq i_{2},i_{3},\ldots,i_{l}\leq n}w_{\nu}\left(\frac{i_{1}-i_{2}}{c_{n}}\right)w_{\nu}\left(\frac{i_{2}-i_{3}}{c_{n}}\right)\cdots w_{\nu}\left(\frac{i_{l}-i_{1}}{c_{n}}\right)\nonumber\\
=&l\lim_{n\to\infty}\frac{1}{c_{n}^{l-1}}\sum_{1\leq i_{2},i_{3},\ldots,i_{l}\leq n}w_{\nu}\left(\frac{i_{1}-i_{2}}{c_{n}}\right)w_{\nu}\left(\frac{i_{2}-i_{3}}{c_{n}}\right)\cdots w_{\nu}\left(\frac{i_{l}-i_{1}}{c_{n}}\right)\nonumber\\
=&l\int_{[0,1/\nu]^{l-1}}w_{\nu}(t_{1}-t_{2})w_{\nu}(t_{2}-t_{3})\cdots w_{\nu}(t_{l}-t_{1})\;dt_{2}\;dt_{3}\;\ldots\;dt_{l}\label{eqn: periodic convolution}\\
=&l w_{\nu}^{(l)}(0)\nonumber\\
=&l\nu\sum_{k\in \Z}\hat{w}_{\nu}(k)^{l}\label{eqn: discrete monomial variance},
\end{align}

where $(l)$ denotes the $l$ fold convolution, and
\begin{align}\label{eqn: definition of periodic fourier transform}
\hat{w}_{\nu}(k)=\int_{-1/2\nu}^{1/2\nu}w_{\nu}(x)e^{2\pi i k x\nu}\;dx=\int_{-1/2}^{1/2}w(x)e^{2\pi i k x\nu}\;dx
\end{align}
is the $k$th Fourier coefficient of $w_{\nu}$. Note that in any $l$ fold convolution, the integral is taken over $l-1$ many variables only (as in \ref{eqn: periodic convolution}). In our context, this can also be explained by the fact that for each chosen index $i_{1}$, we have freedom to choose the indices $i_{2},\ldots,i_{l}$ only. 

\begin{rem}\label{rem: why periodic2}
	The integral \eqref{eqn: periodic convolution} is taken over $[0,1/\nu]^{l-1}$, because $t_{i}\in [0,n/c_{n}]\to [0,1/\nu]$ for all $2\leq i\leq l$. However, the values of the variables $t_{i+1}-t_{i}$ may fall outside $[0,1/\nu]$ in principle. But using the $1/\nu$ periodicity of $w_{\nu}$, we can bring it back to $[0,1/\nu]$. The above calculation does not work for non-periodic band matrices while $c_{n}=\Omega(n)$. 
\end{rem}

\subsection{Case II: $\lim_{n\to\infty}\frac{c_n}{n}=0$}
The band matrix in Definition \ref{defn: band matrix with a variance profile} is periodic. In that case, we compute
\begin{align}
&\lim_{n\to\infty}\frac{c_{n}}{n}\E\left[\CL_{f}(M)\overline{\CL_{f}(M)}\right]\nonumber\\
=&\lim_{n\to\infty}\frac{c_{n}}{n}\E\left[\left\{\sum_{i_{1},\ldots, i_{l}=1}^{n} m_{i_{1}i_{2}}m_{i_{2}i_{3}}\cdots m_{i_{l}i_{1}}\right\}\left\{\sum_{i_{1},\ldots, i_{l}=1}^{n} \overline m_{j_{1}j_{2}}\overline m_{j_{2}j_{3}}\cdots \overline m_{j_{l}j_{1}}\right\}\right]\label{eqn: expansion of M^l first step}\\
=&l\lim_{n\to\infty}\frac{1}{c_{n}^{l-1}}\sum_{1\leq i_{2},i_{3},\ldots,i_{l}\leq n}w_{0}\left(\frac{i_{1}-i_{2}}{c_{n}}\right)w_{0}\left(\frac{i_{2}-i_{3}}{c_{n}}\right)\cdots w_{0}\left(\frac{i_{l}-i_{1}}{c_{n}}\right)\nonumber\\
=&l\int_{\R^{l-1}}w_{0}(t_{1}-t_{2})w_{0}(t_{2}-t_{3})\cdots w_{0}(t_{l}-t_{1})\;dt_{2}\;dt_{2}\;\ldots\;dt_{l}\nonumber\\
=&l\int_{\R^{l-1}}w_{0}(-t_{2})w_{0}(t_{2}-t_{3})\cdots w_{0}(t_{l})\;dt_{2}\;dt_{2}\;\ldots\;dt_{l}\nonumber\\
=&lw_{0}^{(l)}(0)\nonumber\\
=&l\int_{\R}\hat{w}_{0}(t)^{l}\;dt\label{eqn: continuous monomial variance},
\end{align}
where
\begin{align}\label{eqn: definition of non-periodic fourier transform}
\hat{w}_{0}(t)=\int_{\R}e^{2\pi itx}w_{0}(x)\;dx=\int_{-1/2}^{1/2}e^{2\pi itx}w(x)\;dx.
\end{align}

\begin{rem}\label{rem: why periodic}
In the above, if the band matrix was not periodic, we can make the above integration over $\R^{l-1}$ i.e., $l$ fold convolution, by taking $t_{1}$ far off from the origin. We can do that because of $t_{1}\in (0,n/c_{n})\to(0,\infty)$. Thus, the above calculation will also go through for non-periodic band matrices. However, the same approach can not be implemented in \eqref{eqn: periodic convolution}, as in that case $t_{1}\in (0,n/c_{n})\to (0,1/\nu)$. So, we need the matrix to be periodic when $c_{n}=\Omega(n)$.
\end{rem}

\subsection{Covariance kernel of $\tr R_{z}(M)$}\label{subsec: covariance kernel of tr R}

In the view of \eqref{eqn: expectation of trace of M^l are zero}, we notice that if $k\neq l$ then 
\begin{align*}
\frac{c_{n}}{n}\E\left[\tr M^{k}\tr M^{*l}\right]=
O\left((|l-k|^{18}n^{-1}\nu^{-3})^{|k-l|/6}\right).
\end{align*}
If $k=l=1$, then
\begin{align*}
\lim_{n\to\infty}\frac{c_{n}}{n}\E\left[\tr M\tr M^{*}\right]=w_{\nu}(0).
\end{align*}

Therefore using \eqref{eqn: discrete monomial variance}, Lemma \ref{lem: almost sure bound on norm} for $z,\eta>\rho$, and proceeding as \eqref{eqn: trace of resolvent expansion}, \eqref{eqn: expectation of tr M^l on the complement event} on page \pageref{eqn: trace of resolvent expansion} we have,
\begin{align*}
&\lim_{n\to\infty}\frac{c_{n}}{n}\E[\tr \hat R_{z}^{\circ}(M)\tr \hat R_{\bar{\eta}}^{\circ}(M^{*})]\\
=&\lim_{n\to\infty}\frac{c_{n}}{n}\left\{\E[\tr \hat R_{z}(M)\tr \hat R_{\bar{\eta}}(M^{*})]-\frac{n^{2}}{z\bar{\eta}}\right\}\\
=&\lim_{n\to\infty}\frac{c_{n}}{n}\sum_{l=1}^{\infty}(z\bar{\eta})^{-l-1}\E\left[\tr M^{l}\tr M^{*l}\right]\\
=&\frac{w_{\nu}(0)}{(z\bar{\eta})^{2}}+\nu\sum_{l=2}^{\infty}\sum_{k\in \Z}l\frac{\hat{w}_{\nu}(k)^{l}}{(z\bar{\eta})^{l+1}}\\
=&\frac{w_{\nu}(0)}{(z\bar{\eta})^{2}}+\nu\sum_{k\in \Z}\hat{w}_{\nu}(k)(z\bar{\eta})^{-2}\left[\left(1-\frac{\hat{w}_{\nu}(k)}{z\bar{\eta}}\right)^{-2}-1\right]\\
=&\nu\sum_{k\in \Z}\frac{\hat{w}_{\nu}(k)}{(z\bar{\eta}-\hat{w}_{\nu}(k))^{2}},
\end{align*}
where the last expression follows from the fact that $w_{\nu}(0)=\nu\sum_{k\in \Z}\hat{w}_{\nu}(k)$. If we take $\nu\downarrow 0$ in the above expression, we obtain
\begin{align*}
\int_{\R}\frac{\hat{w}_{0}(t)}{(z\bar{\eta}-\hat{w}_{0}(t))^{2}}\;dt.
\end{align*}

Alternatively when $z,\eta>\rho$, using \eqref{eqn: continuous monomial variance},
\begin{align*}
&\lim_{n\to\infty}\frac{c_{n}}{n}\left\{\E[\tr \hat R_{z}(M)\overline{\tr \hat R_{\eta}(M)}]-\frac{n^{2}}{z\bar{\eta}}\right\}\\
=&\sum_{l=1}^{n}(z\bar{\eta})^{-l-1}l\int_{\R}\hat{w}_{0}(t)^{l}\;dt\\
=&\int_{\R}\frac{\hat{w}_{0}(t)}{(z\bar{\eta})^{2}}\left(1-\frac{\hat{w}_{0}(t)}{z\bar{\eta}}\right)^{-2}\;dt\\
=&\int_{\R}\frac{\hat{w}_{0}(t)}{(z\bar{\eta}-\hat{w}_{0}(t))^{2}}\;dt.
\end{align*}

\subsection{Connection to Irwin-Hall distribution \& Eulerian numbers}\label{subsec: irwin hall distribuion} Suppose the weight profile $w(x)\equiv 1$. Then \eqref{eqn: expansion of M^l first step} can be written as $l\gamma_{l}$, where
\begin{align*}
\gamma_{l}:=\Pb(|U_{2}+U_{3}+\cdots+U_{l}|\leq 1/2),
\end{align*}

and $U_{i}\stackrel{i.i.d.}{\sim}\text{Unif}[-1/2,1/2]$. Let $S_{l-1}=\sum_{i=2}^{l}U_{i}$ and  $p_{l-1}(x)$ be the pdf of $S_{l-1}$. Since $S_{l-1}+\frac{l-1}{2}$ follows the Irwin-Hall distribution, the density of $S_{l-1}$ is given by
\begin{align*}
p_{l-1}(x)=\frac{1}{(l-2)!}\sum_{i=0}^{\lfloor x+(l-1)/2\rfloor}(-1)^{i}\binom{l-1}{i}\left(x+\frac{l-1}{2}-i\right)^{l-2},
\;x\in \left[-\frac{l-1}{2},\frac{l-1}{2}\right].
\end{align*}  
For a geometric derivation of the above formula, see \cite{marengo2017geometric}. On the other hand, the characteristic function of $S_{l-1}$ is given by
\begin{align*}
\E[e^{itS_{l-1}}]=\left\{\E[e^{itU_{1}}]\right\}^{l-1}=\left(\sinc(t/2)\right)^{l-1}.
\end{align*}
Using the inversion formula, 
\begin{align*}
p_{l-1}(x)=\frac{1}{2\pi}\int_{\R}e^{-itx}\sinc^{l-1}(t/2)\;dt=\frac{1}{\pi}\int_{\R}e^{-itx}\sinc^{l-1}(t)\;dt.
\end{align*}
Therefore,
\begin{align*}
\gamma_{l}&=\Pb(|S_{l-1}|\leq 1/2)
=\int_{-1/2}^{1/2}p_{l-1}(x)\;dx
=\frac{1}{\pi}\int_{\R}\sinc^{l}(t)\;dt
=p_{l}(0)\\
&=\frac{1}{(l-1)!}\sum_{i=0}^{\lfloor l/2\rfloor}(-1)^{i}\binom{l}{i}\left(\frac{l}{2}-i\right)^{l-1}=\frac{A(l-1, l/2-1)}{(l-1)!},
\end{align*}
where $A(l-1,l/2-1)$ is an Eulerian number for even $l$. $A(n,m)$ counts the number of permutations of $1,2,\ldots,n$ in which exactly $m$ elements are bigger than the previous element. The above establishes \eqref{eqn: sinc integration is Irwin-Hall}.


\appendix
\section{}\label{sec: norm of the matrix}
In this section we discuss about norm of random non-Hermitian matrices. A sharp almost sure bound on the spectral radius of non-Hermitian random matrices can be found in \cite{geman1980limit, geman1986spectral}. 

\begin{thm}\cite{geman1986spectral}\label{thm: sharp as bound for full matrix norm}
Let $M=(m_{ij})_{n\times n}$ be a sequence of $n\times n$ random matrices with $m_{ij};1\leq i,j\leq n$ real valued i.i.d. for each $n$. Assume that for each $n$,
\begin{enumerate}[(i)]
	\item $\E[m_{11}]=0$
	\item $\E[m_{11}^{2}]=\sigma^{2}$
	\item $\E[|m_{11}|^{p}]\leq p^{c p}$ for all $p\geq 2$ and for some $c>0$.
\end{enumerate}
Let 
\begin{align*}
\rho_{n}=\max_{1\leq i\leq n}\{|\lambda_{i}(M/\sqrt{n})|\}.
\end{align*}
Then $\limsup_{n\to\infty}\rho_{n}\leq \sigma$ almost surely.
\end{thm}
We see that in our context if a full random matrix satisfies the condition \ref{con: The condition}, then it also satisfies the above condition and as a result, $\|M\|\leq 1$ almost surely as $n\to\infty$. However, the above theorem does not take a variance profile into account. The following theorem from \cite{bandeira2016sharp} estimates the norm of a symmetric random matrix with a variance profile.

\begin{thm}\cite[Corollary 3.5]{bandeira2016sharp}\label{thm: expectation bound for matrix norm}
	Let $X$ be a $n\times n$ real symmetric matrix with $X_{ij}=\xi_{ij}w_{ij}$, where $\{\xi_{ij}:i\geq j\}$ are independent centered random variables and $\{w_{ij}:i\geq j\}$ are give scalars. If $\E[|\xi_{ij}|^{2p}]^{1/2p}\leq C p^{\beta/2}$ for some $C,\beta>0$ and all $p,i,j$, then
	\begin{align*}
	\E\|X\|\leq C'\max_{i}\sqrt{\sum_{j=1}^{n}w_{ij}^{2}}+C'\max_{i,j}|w_{ij}|\log^{(\beta\wedge 1)/2}n,
	\end{align*}
	where $C$ depends on $C,\beta$ only.
\end{thm}

Using the above theorem along with Poincar\'e inequality, we have the following lemma.
\begin{lem}\label{lem: almost sure bound on norm}
	Let $M$ be an $n\times n$ random matrix as in the Theorem \ref{thm: main theorem}. Then there exists $\rho\geq 1$ such that
	\begin{align}
	\Pb\left(\|M\|>\rho/4+t\right)\leq K\exp\left(-\sqrt{\frac{\alpha c_{n}}{2\omega}}t\right),\;\;\forall\;t>0,\label{eqn: tail bound on norm}
	\end{align}
	where $K>0$ is a universal constant. In particular,
	\begin{align}
	&\Pb(\|M\|>\rho\;\text{infinitely often})=0,\label{eqn: almost sure bound on norm}\\
	\text{and}\;& \;\E[\|M\|^{l}]\leq \rho^{l}+K\Gamma(l+1)\left(\frac{\alpha c_{n}}{2\omega}\right)^{-l/2}\exp\left(\sqrt{\frac{\alpha c_{n}}{2\omega}}\frac{\rho}{4}\right).\label{eqn: bound on expectation of norm}
	\end{align}
\end{lem}

\begin{proof}
First of all, we may write $M=M_{R}+iM_{I}$, where both $M_{R}$ and $M_{I}$ are real valued matrices. Then we can estimate $\|M\|\leq \|M_{R}\|+\|M_{I}\|$. Therefore without loss of generality, let us consider $M$ be a real valued matrix. Consider
\begin{align*}
\tilde{M}:=\left[\begin{array}{cc}
O & M\\
M^{*} & O
\end{array}
\right],
\end{align*}
and apply theorem \ref{thm: expectation bound for matrix norm} on $\tilde M$ to obtain the same bound for $\E\|\tilde{M}\|(=\E\|M\|)$. In our case, $w_{ij}^{2}=\frac{1}{c_{n}}w_{\nu}((i-j)/n)$ or $w_{ij}^{2}=\frac{1}{c_{n}}w_{0}((i-j)/n)$ as described in Definition \ref{defn: band matrix with a variance profile}. Since $c_{n}\geq \log^{3}n$ and $w$ is a piece-wise continuous function, $\lim_{n\to\infty}\sum_{j=1}^{n}w_{ij}^{2}=\int w(x)\;dx=1$  and $\max_{i,j}|w_{ij}|\log^{(\beta\wedge 1)/2}n\to 0$. As a result, there exists $\rho\geq 1$ such that $$\limsup\E\|M\|\leq \rho/4.$$ Here we note that we need $c_{n}$ to grow at least as $\log n$. In fact, this is a sharp condition. Otherwise, the matrix norm may be unbounded \cite{bogachev1991level, khorunzhy2004spectral}.

Secondly, $\|M\|\leq \sqrt{\sum_{i,j}|m_{ij}|^{2}}$ implies that $h(M):= \|M\|$ is a Lipschitz$_{1}$ function. Therefore applying the properties of Poincar\'e inequality as described in Definition \ref{defn: poincare inequality}, we have
\begin{align*}
&\Pb\left(|\|M\|-\E\|M\||>t\right)\leq K\exp\left(-\sqrt{\frac{\alpha c_{n}}{2\omega}}t\right)\\
\text{i.e.,} \;&\;\Pb\left(\|M\|>\rho/4+t\right)\leq K\exp\left(-\sqrt{\frac{\alpha c_{n}}{2\omega}}t\right),
\end{align*}
where $\omega=\sup_{x}w(x)$.  

Equation \eqref{eqn: almost sure bound on norm} can be seen from the equation \eqref{eqn: tail bound on norm} along with the application of Borel-Cantelli lemma with $t=3\rho/4$. Equation \eqref{eqn: bound on expectation of norm} can be justified as follows,
\begin{align*}
\E[\|M\|^{l}]&=\int_{0}^{\rho/4}lu^{l-1}\Pb(\|M\|>u)\;du+\int_{\rho/4}^{\infty}lu^{l-1}\Pb(\|M\|>u)\;du\\
&\leq \rho^{l}+Kl\exp\left(\sqrt{\frac{\alpha c_{n}}{2\omega}}\frac{\rho}{4}\right)\int_{\rho/4}^{\infty}u^{l-1}\exp\left(-\sqrt{\frac{\alpha c_{n}}{2\omega}}u\right)\;du\\
&\leq \rho^{l}+K\Gamma(l+1)\left(\frac{\alpha c_{n}}{2\omega}\right)^{-l/2}\exp\left(\sqrt{\frac{\alpha c_{n}}{2\omega}}\frac{\rho}{4}\right).
\end{align*}


 
 
 We would like to remark that the lemma is also true for column removed matrices $M^{(k)}$, and the same proof will go through.
\end{proof}

We finally would like to remark that although the Theorem \ref{thm: expectation bound for matrix norm} gives a constant bound on the norm of the matrix with a variance profile, the constant is not that sharp unlike Theorem \ref{thm: sharp as bound for full matrix norm}. However, we expect that for matrices with continuous variance profile, the correct norm bound should be $\limsup_{n\to\infty}\sum_{i=1}^{n}w_{ij}^{2}$. This was remarked in \cite[Remark 4.11]{latala2018dimension}. In our case, this limit is equal to $1$. As we have mentioned in remark \ref{rem: eventually it is 1} that eventually it suffices to take $z,\eta\in \partial \D_{1}$ only.

\section{}

Here we list down the two key ingredients; martingale difference CLT, and Sherman-Morrison formula. Interested readers may find the proofs in the included references.

\begin{lem}\cite[Theorem 35.12]{billingsley2008probability}\label{lem: martingale difference CLT}
Let $\{\xi_{n,k}\}_{1\leq k\leq n}$ be a martingale difference array with respect to a filtration $\{\CF_{k,n}\}_{1\leq k\leq n}$. Suppose for any $\delta>0$,
\begin{enumerate}[(i)]
	\item $\lim_{n\to\infty}\sum_{k=1}^{n}\E[\xi_{n,k}^{2}\1_{|\xi_{n,k}|>\delta}]=0$,
	\item $\sum_{k=1}^{n}\E[\xi_{n,k}^{2}|\CF_{n,k-1}]\stackrel{p}{\to}\sigma^{2}$ as $n\to\infty$.
\end{enumerate}
Then $\sum_{k=1}^{n}\xi_{n,k}\stackrel{d}{\to}\CN_{1}(0,\sigma^{2})$.
\end{lem}

\begin{lem}[\cite{sherman1950adjustment}, Sherman-Morrison formula]\label{lem: Sherman-Morrison formula}
	Let $A$ and $A+ve_{k}^{t}$ be two invertible matrices, where $v\in \C^{n}$. Then
	\begin{align*}
	(A+ve_{k}^{t})^{-1}v=\frac{A^{-1}v}{1+e_{k}^{t}A^{-1}v}.
	\end{align*}
\end{lem}


\bibliographystyle{abbrv}

\end{document}